\documentclass[10pt,reqno]{amsart}
\usepackage{amssymb,mathrsfs,graphicx,amsmath,mathtools}
\usepackage{ifthen}
\usepackage{hyperref}
\usepackage[margin=1in]{geometry}
\usepackage{caption}
\usepackage{sidecap}
\usepackage{rotating}
\usepackage{ifthen,latexsym,float,colortbl}

\provideboolean{shownotes} 
\setboolean{shownotes}{true} 
\newcommand{\margnote}[1]{
\ifthenelse{\boolean{shownotes}}%
{\marginpar{\raggedright\tiny\texttt{#1}}}%
{}%
}
\newcommand{\hole}[1]{
\ifthenelse{\boolean{shownotes}}%
{\begin{center} \fbox{ \rule {.25cm}{0cm} \rule[-.1cm]{0cm}{.4cm}
\parbox{.85\textwidth}{\begin{center} \texttt{#1}\end{center}} \rule
{.25cm}{0cm}}\end{center}} {} }

\title[Mean-field limits: from particle descriptions to macroscopic equations]{Mean-field limits: from particle descriptions to macroscopic equations}

\author[Carrillo]{Jos\'{e} A. Carrillo}
\address[Jos\'{e} A. Carrillo]{\newline Department of Mathematics
    \newline Mathematical Institute, University of Oxford, Oxford OX2 6GG, UK}
\email{carrillo@maths.ox.ac.uk}

\author[Choi]{Young-Pil Choi}
\address[Young-Pil Choi]{\newline Department of Mathematics \newline
Yonsei University, 50 Yonsei-Ro, Seodaemun-Gu, Seoul 03722, Republic of Korea}
\email{ypchoi@yonsei.ac.kr}

\numberwithin{equation}{section}

\newtheorem{theorem}{Theorem}[section]
\newtheorem{lemma}{Lemma}[section]

\newtheorem{proposition}{Proposition}[section]
\newtheorem{remark}{Remark}[section]
\newtheorem{definition}{Definition}[section]

\newcommand{\R}{\mathbb R}

\newcommand{\mc}{\mathcal C}

\newcommand{\bq}{\begin{equation}}
\newcommand{\eq}{\end{equation}}
\newcommand{\e}{\varepsilon}
\newcommand{\lt}{\left}
\newcommand{\rt}{\right}

\newcommand{\pa}{\partial}

\newcommand{\mz}{\mathcal{Z}}
\newcommand{\me}{\mathcal{E}}
\newcommand{\mf}{\mathcal{F}}

\newcommand{\W}{\mathcal{W}}
\newcommand{\intr}{\int_{\R^d}}
\newcommand{\intrr}{\int_{\R^d \times \R^d}}
\newcommand{\intb}{\int_{B(0,R)}}
\newcommand{\wt}{\widetilde}

\begin{document}
\allowdisplaybreaks


\subjclass[]{}

\begin{abstract} We rigorously derive pressureless Euler-type equations with nonlocal dissipative terms in velocity and aggregation equations with nonlocal  velocity fields from Newton-type particle descriptions of swarming models with alignment interactions.  We crucially make use of a discrete version of a modulated kinetic energy together with the bounded Lipschitz distance for measures in order to control terms in its time derivative due to the nonlocal interactions.
\end{abstract}

\maketitle \centerline{\date}



%
%
%
%
\section{Introduction}
In this work, we analyse the evolution of an indistinguishable $N$-point particle system given by
\begin{align}\label{main_par}
\begin{aligned}
\dot x_i &= v_i, \quad i=1,\dots,N, \quad t > 0,\cr
\e_N \dot v_i & =-\gamma v_i - \nabla_x V(x_i) - \frac1N \sum_{j=1}^N \nabla_x W(x_i - x_j)+ \frac1N \sum_{j=1}^N \psi(x_i - x_j)(v_j - v_i)
\end{aligned}
\end{align}
subject to the initial data
\bq\label{ini_main_par}
(x_i,v_i)(0) =: (x_i(0), v_i(0)), \quad i =1,\dots,N.
\eq
Here $x_i = x_i(t) \in \R^d$ and $v_i = v_i(t) \in \R^d$ denote the position and velocity of $i$-particle at time $t$, respectively. The coefficient $\gamma \geq 0$ represents the strength of linear damping in velocity, $\e_N>0$ the strength of inertia, $V:\R^d \to \R_+$ and $W : \R^d \to \R$ represent the confinement and interaction potentials, respectively. $\psi : \R^d \to \R_+$ is a communication weight function. Throughout this paper, we assume that $W$ and $\psi$ satisfy $W(x) = W(-x)$ and $\psi(x) = \psi(-x)$ for $x \in \R^d$. They include basic particle models for collective behavior, see \cite{CCP17, CHL17, DCBC,CS07,CFTV, HL09, HT08} and the references therein.

Our main goal is to derive the macroscopic collective models rigorously governing the evolution of the particle system \eqref{main_par} as the number of particles goes to infinity. On one hand, we will derive hydrodynamic Euler-alignment models given by  
\begin{align}\label{main_fluid}
\begin{aligned}
&\pa_t \rho + \nabla_x \cdot (\rho u) = 0,\cr
&\pa_t (\rho u) + \nabla_x \cdot (\rho u \otimes u) = -\gamma \rho u - \rho \nabla_x V - \rho \nabla_x W \star\rho + \rho \int_{\R^d} \psi(x-y) (u(y) - u(x))\,\rho(y)\,dy
\end{aligned}
\end{align}
in the mean-field limit: when initial particles are close to a monokinetic distribution $\rho_0(x) \delta_{u_0(x)}(v)$ in certain sense and $\e_N=O(1)$ as $N\to\infty$. On the other hand, we will show that the particle system can be described by aggregation equations of the form
\bq\label{eq_agg}
\pa_t \bar \rho + \nabla_x \cdot (\bar \rho \bar u) = 0,
\eq
where
\bq\label{eq_agg2}
\gamma \bar \rho \bar u = - \bar  \rho \nabla_x V - \bar  \rho \nabla_x W \star \bar  \rho +\bar  \rho \intr \psi(x-y)(\bar  u(y) - \bar  u(x)) \bar \rho(y)\,dy
\eq
in the combined mean-field/small inertia limit: when initial particles are close to a monokinetic distribution $\rho_0(x) \delta_{u_0(x)}(v)$, $\gamma>0$ and $\e_N \to 0$ as $N \to \infty$. For simplicity of notation when dealing with the mean-field limit, we will take $\e_N=1$ in the sequel.

\subsection{Mean-field limits: from particles to continuum}
As the number of particles $N$ tends to infinity, microscopic descriptions given by the particle system \eqref{main_par} become more and more computationally unbearable. Reducing the complexity of the system is of paramount importance in any practical application. The classical multiscale strategy in kinetic modelling is to introduce the number density function $f = f(x,v,t)$ in phase space $(x,v) \in \R^d \times \R^d$ at time $t \in \R_+$ and study the time evolution of that density function. Then at the formal level, we can derive the following Vlasov-type equation from the particle system \eqref{main_par} as $N \to \infty$:
\bq\label{main_kin}
\pa_t f + v \cdot \nabla_x f - \nabla_v \cdot \lt((\gamma v + \nabla_x V + \nabla_x W \star \rho_f )f\rt) + \nabla_v \cdot (F_a(f)f) =0,
\eq
where $\rho_f = \rho_f(x,t)$ is the local particle density and $F_a(f) = F_a(f)(x,v,t)$ represents a nonlocal velocity alignment force given by
\[
\rho_f(x,t) := \intr f(x,v,t)\,dv
\]
and
\[
F_a(f)(x,v,t) := \intrr \psi(x-y)(w-v)f(y,w,t)\,dydw,
\]
respectively. Let us briefly recall the reader the basic formalism leading to the kinetic equation \eqref{main_kin} as the limiting system of \eqref{main_par}. We first define the empirical measure $\mu^N$ associated to a solution to the particle system \eqref{main_par}, i.e.,
\[
\mu^N_t(x,v) := \frac1N \sum_{i=1}^N \delta_{(x_i(t), v_i(t))}.
\]
As long as there exists a solution to \eqref{main_par}, the empirical measure $\mu^N$ satisfies \eqref{main_kin} in the sense of distributions. To be more specific, for any $\varphi \in \mc^1_0(\R^d \times \R^d)$, we get
\begin{equation}\label{aux}\begin{aligned}
\frac{d}{dt}\intr \varphi(x,v)\,\mu^N_t(dxdv)&= \frac{d}{dt} \frac1N \sum_{i=1}^N \varphi(x_i(t), v_i(t))\cr
&= \frac1N \sum_{i=1}^N \lt(\nabla_x \varphi(x_i(t), v_i(t)) \cdot v_i(t) + \nabla_v \varphi(x_i(t), v_i(t)) \cdot \dot{v_i}(t) \rt).
\end{aligned}\end{equation}
Notice that the particle velocity can also be rewritten in terms of the empirical measure $\mu^N$ as
\[
\dot{v_i}(t)= \gamma v_i + \nabla_x V(x_i) + \intrr \nabla_x W(x_i - y)\,\mu^N_t(dydw) + \intrr \psi(x_i - y)(w-v_i)\,\mu^N_t(dydw).
\]
This implies that the right-hand side of \eqref{aux} can also be written in terms of the empirical measure $\mu^N$ as
$$\begin{aligned}
\frac{d}{dt}\intr \varphi(x,v) &\, \mu^N_t(dxdv)
= \intrr \nabla_x \varphi(x,v)\,\mu^N_t(dxdv)\cr
&  -\intrr \nabla_v \varphi(x,v) \cdot \lt(\gamma v + \nabla_x V(x) +\intrr \nabla_x W(x - y)\,\mu^N_t(dydw) \rt)\mu^N_t(dxdv)\cr
& + \intrr \nabla_v \varphi(x, v) \cdot \lt(\intrr \psi(x - y) (w-v)\,\mu^N_t(dydw)\rt)\mu^N_t(dxdv).
\end{aligned}$$
This concludes that $\mu^N$ is a solution to \eqref{main_kin} in the sense of distributions as long as particle paths are well defined. In fact, if the interaction potential $W$ and the communication weight function $\psi$ are regular enough, for instance, bounded Lipschitz regularity, then the global-in-time existence of measure-valued solutions can be obtained by establishing a weak-weak stability estimate for the empirical measure, see \cite[Section 5]{HL09} for more details. The mean-field limit has attracted lots of attention in the last years in different settings depending on the regularity of the involved potentials $V,W$ and communication function $\psi$. Different approaches to the derivation of the Vlasov-like kinetic equations with alignments/interaction terms or the aggregation equations have been taken leading to a very lively interaction between different communities of researchers in analysis and probability. We refer to \cite{BH,dobru, Neun,Spohn,G03,HT08,CFTV,BCC11,CCHS19,CS18,CS19,H14,JW16,JW17,JW18} for the classical references and non-Lipschitz regularity velocity fields in kinetic cases, to \cite{H09,HIpre} for very related incompressible fluid problems, and to \cite{HJ07,CDFLS11,CCH14,FHM14,HJ15,Deu16,G16,LP17,PS17,CDP20,Spre,BJW20} for results with more emphasis on the singular interaction kernels both at the kinetic and the aggregation-diffusion equation cases.


\subsection{Local balanced laws, the mono-kinetic ansatz, and the large friction limit} We introduce several macroscopic observables; local momentum $\rho_f u_f : \R^d \times \R_+ \to \R^d$, local energy $\rho_f E_f: \R^d \times \R_+ \to \R_+$, strain tensor $P_f: \R^d \times \R_+ \to \R^d \times \R^d$, and heat flux $q_f : \R^d \times \R_+ \to \R^d$ defined as
\[
\rho_f u_f := \intr vf\,dv, \quad \rho_f E_f := \intr |v|^2 f\,dv, \quad P_f := \intr (u-v)\otimes (u-v) f\,dv,
\]
and
\[
q_f := \intr |v-u|^2(v-u)f\,dv.
\]
Here $\cdot \otimes \cdot$ stands for $(a \otimes b)_{ij} = a_i b_j$ for $a = (a_1,\dots,a_d) \in \R^d$ and $b = (b_1,\dots,b_d) \in \R^d$. Then, at the formal level, by taking moments of the kinetic equation \eqref{main_kin}, one can derive the following system of local balanced laws:
\begin{align}\label{eq_cons}
\begin{aligned}
&\pa_t \rho_f + \nabla_x \cdot (\rho_f u_f) =0, \quad (x,t) \in \R^d \times \R_+,\cr
&\pa_t (\rho_f u_f) + \nabla_x \cdot (\rho_f u_f \otimes u_f) + \nabla_x \cdot P_f \cr
&\qquad = - \gamma \rho_f u_f - \rho_f \nabla_x V - \rho_f \nabla_x W \star \rho_f + \rho_f\int_{\R^d} \psi(x-y)(u_f(y) - u_f(x))\rho_f(y)\,dy, \cr
&\pa_t (\rho_f E_f) + \nabla_x \cdot \lt(\rho_f E_f u_f + P_fu_f + q_f \rt)\cr
&\qquad = -\gamma \rho_f E_f - \rho_f u_f \cdot \nabla_x V - \rho_f u_f \cdot \nabla_x W \star\rho_f \cr
&\qquad \quad + \rho_f \intr \psi(x-y) \lt(u_f(x) \cdot u_f(y) - E_f(x) \rt) \rho_f(y)\,dy.
\end{aligned}
\end{align}
The system \eqref{eq_cons} is not closed. Suitable closure assumptions are not known so far even in cases where noise/diffusion is added to the system. However, at the formal level, we can take into account the mono-kinetic ansatz for $f$, as done in \cite{CDP09,CKMT10}, leading to
\bq\label{mono-kin}
f (x,v,t) \simeq \rho_f(x,t) \delta_{u_f(x,v)}(v).
\eq
Then the strain tensor and heat flux become zero and the system \eqref{eq_cons} closes becoming the pressureless Euler equations with nonlocal interaction forces \eqref{main_fluid}:
$$\begin{aligned}
&\pa_t \rho + \nabla_x \cdot (\rho u) =0, \quad (x,t) \in \R^d \times \R_+,\cr
&\pa_t u + u \cdot\nabla_x  u = - \gamma  u - \nabla_x V - \nabla_x W \star \rho + \int_{\R^d} \psi(x-y)(u(y) - u(x))\rho(y)\,dy, \cr
&\pa_t |u|^2 + u\cdot \nabla_x |u|^2= -\gamma  |u|^2 -  u \cdot \nabla_x V - u \cdot \nabla_x W \star\rho +  \intr \psi(x-y) \lt(u_f(x) \cdot u(y) - |u(x)|^2 \rt) \rho(y)\,dy
\end{aligned}$$
on the support of $\rho$. Notice that we have eliminated the subscript $_f$ in the hydrodynamic quantities since the system \eqref{main_fluid} is now closed, the last equation is redundant but it gives a nice information about the total energy of the system. Although the monokinetic assumption is not fully rigorously justified and it does not have a direct physical motivation, it is observed by particle simulations that the derived hydrodynamic system shares some qualitative behavior with the particle system, see \cite{CDMBC,CDP09,CFTV,CKMT10,CKR16,CCP17}. Note that \eqref{main_fluid} conserves only the total mass in time in this generality. However, the total free energy is dissipated due to the linear damping and the velocity alignment force as pointed out in \cite{CFGS17} for weak solutions of this system. The hydrodynamic system \eqref{eq_cons} has a rich variety of phenomena due to the competition between attraction/repulsion and alignment leading to sharp thresholds for the global existence of strong solutions versus finite time blow-up and decay to equilibrium, see \cite{TT14,CCTT16,CCZ16,CCT19,CH19}.

It is worth noticing as in \cite{CDP09} that the mono-kinetic ansatz for $f$ is a measure-valued solution of the kinetic equation \eqref{main_kin}. More precisely, one can show that $\rho(x,t) \delta_{u(x,t)}(v)$ is a solution to the kinetic equation \eqref{main_kin} in the sense of distributions as long as $(\rho,u)(x,t)$ is a strong solution to the hydrodynamic equations \eqref{main_fluid}. Indeed, for any $\varphi \in \mc^1_0(\R^d \times \R^d)$, we obtain
$$\begin{aligned}
&\frac{d}{dt}\intrr \varphi(x,v) \rho(x,t) \,\delta_{u(x,t)}(dv)\,dx\cr
&\quad = \frac{d}{dt}\intr \varphi(x,u(x,t)) \rho(x,t) \,dx = \intr \varphi(x,u(x,t)) \pa_t \rho\,dx + \intr (\nabla_v \varphi) (x, u(x,t)) \cdot (\pa_t u) \rho\,dx =: I_1 + I_2.
\end{aligned}$$
Using the continuity equation in \eqref{main_fluid}, $I_1$ can be easily rewritten as
$$\begin{aligned}
I_1 &= \intr \nabla_x (\varphi(x, u(x,t))) \cdot (\rho u)\,dx \cr
&= \intrr (\nabla_x \varphi)(x,v)\cdot (\rho v) \delta_{u(x,t)}(dv)\,dx+ \intr (\nabla_v \varphi)(x,u(x,t)) \cdot \rho (u \cdot \nabla_x) u\,dx.
\end{aligned}$$
By multiplying the velocity equation in \eqref{main_fluid} by $\rho$ and using $(\nabla_v \varphi) (x, u(x,t))$ as a test function to the resulting equation yields
$$\begin{aligned}
I_2 &= - \intr (\nabla_v \varphi) (x, u(x,t)) \cdot (\pa_t u) \rho\,dx - \intr (\nabla_v \varphi) (x, u(x,t)) \cdot \lt(\gamma u + \nabla_x V + \nabla_x W \star\rho\rt)\rho\,dx\cr
&\quad + \intrr (\nabla_v \varphi) (x, u(x,t)) \cdot (u(y) - u(x)) \psi(x-y)\rho(x)\rho(y)\,dxdy.
\end{aligned}$$
Then similarly as before, we can rewrite the second and third terms on the right hand side of the equality by using the mono-kinetic ansatz \eqref{mono-kin}. This implies
$$\begin{aligned}
I_2 &= - \intr (\nabla_v \varphi) (x, u(x,t)) \cdot (\pa_t u) \rho\,dx - \intr (\nabla_v \varphi) (x, v) \cdot \lt(\gamma v + \nabla_x V + \nabla_x W \star\rho\rt)\rho \delta_{u(x,t)}(dv)\,dx\cr
&\quad + \intrr (\nabla_v \varphi) (x, v) \cdot (w - v) \psi(x-y)\rho(x)\delta_{u(x,y)}(dv)\rho(y)\delta_{u(y,t)}(dw)\,dxdy.
\end{aligned}$$
Combining all of the above estimates yields
$$\begin{aligned}
\frac{d}{dt}\intrr \varphi(x,v) \,&\rho(x,t) \delta_{u(x,t)}(dv)\,dx= \intrr ((\nabla_x \varphi)(x,v)\cdot v) \rho \delta_{u(x,t)}(dv)\,dx\cr
&- \intr (\nabla_v \varphi) (x, v) \cdot \lt(\gamma v + \nabla_x V + \nabla_x W \star\rho\rt)\rho \delta_{u(x,t)}(dv)\,dx\cr
& + \intrr (\nabla_v \varphi) (x, v) \cdot (w - v) \psi(x-y)\rho(x)\delta_{u(x,y)}(dv)\rho(y)\delta_{u(y,t)}(dw)\,dxdy.
\end{aligned}$$
This shows that $\rho(x,t) \delta_{u(x,t)}(v)$ satisfies the kinetic equation \eqref{main_kin} in the sense of distributions. 

Finally, we will be also dealing with the small inertia limit for both the kinetic equation \eqref{main_kin} and the hydrodynamic system \eqref{main_fluid} combined with the mean fieild limit. In the small inertia asymptotic limit, we want to describe the behavior of the scaled kinetic equation
\bq\label{main_kin2}
\e(\pa_t f + v \cdot \nabla_x f) - \nabla_v \cdot \lt((\gamma v + \nabla_x V + \nabla_x W \star \rho_f )f\rt) + \nabla_v \cdot (F_a(f)f) =0,
\eq
and the scaled hydrodynamic system
\begin{align}\label{main_fluid1_2}
\begin{aligned}
&\pa_t \rho + \nabla_x \cdot (\rho u) = 0,\cr
&\e(\pa_t (\rho u) + \nabla_x \cdot (\rho u \otimes u)) = -\gamma \rho u - \rho \nabla_x V - \rho \nabla_x W \star\rho + \rho \int_{\R^d} \psi(x-y) (u(y) - u(x))\,\rho(y)\,dy, 
\end{aligned}
\end{align}
in the limit of small inertia $\e\to 0$. At the formal level, the equations \eqref{main_fluid1_2} will be replaced by \eqref{eq_agg}--\eqref{eq_agg2}
as $\e \to 0$. The limiting nonlinearly coupled aggregation equations \eqref{eq_agg}--\eqref{eq_agg2} have been recently studied in \cite{FS15,FST16}. Several authors have studied particular choices of interactions $V,W$ and comunication functions $\psi$ for some of the connecting asymptotic limits from the kinetic description \eqref{main_kin2} with/without noise to the hydrodynamic system \eqref{main_fluid1_2} in \cite{KMT15,FK19,CCpre,CCJpre}, from the hydrodynamic system \eqref{main_fluid1_2} to the aggregation equation \eqref{eq_agg}--\eqref{eq_agg2} in \cite{LT13,LT17,CPWpre}, and for the direct limit from the kinetic equation to the aggregation equation \eqref{eq_agg}--\eqref{eq_agg2} in \cite{Jab00,CCpre}.


\subsection{Purpose, mathematical tools and main novelties}
Summarizing the main facts of the mean-field limit and the monokinetic ansatz in Subsections 1.1 and 1.2, both the empirical measure $\mu^N(t)$ associated to the particle system \eqref{main_par} and the monokinetic solutions $\rho(x,t) \delta_{u(x,t)}$, with $(\rho,u)(x,t)$ satisfying the hydrodynamic equations \eqref{main_fluid} in the strong sense, are distributional solutions of the same kinetic equation \eqref{main_kin}. In order to analyse the convergence of the empirical measure $\mu^N$ to $\rho(x,t) \delta_{u(x,t)}$, the goal is to establish a weak-strong stability estimate where the strong role is played by the distributional solution $\rho(x,t) \delta_{u(x,t)}$ associated to the strong solution of the hydrodynamic system \eqref{main_fluid}. Our main goal is then to quantify the following convergence
\[
\mu^N_t(x,v) \to \rho(x,t) \delta_{u(x,t)}(v) \quad \mbox{as} \quad N \to \infty
\]
in the sense of distributions for both the mean-field and the combined mean-field/small inertia limit for well prepared initial data. Our main mathematical tools are the use of a modulated kinetic energy combined with the bounded Lipschitz distance in order to control terms between the discrete particle system and the hydrodynamic quantities. Let us first introduce the modulated kinetic energy as
\bq\label{energy_mod}
\frac12 \intrr f|v - u|^2\,dxdv,
\eq
where $f$ is a solution of kinetic equation \eqref{main_kin} and $u$ is the velocity field as part of the solution of the pressureless Euler equations \eqref{main_fluid}. We would like to emphasize that the quantity \eqref{energy_mod} gives a sharper estimate compared to the classical modulated macroscopic energy. Indeed, the macro energy of the system \eqref{main_fluid} is given by
\[
E(U):= \frac{|m|^2}{2 \rho} \quad \mbox{with} \quad  U:= \begin{pmatrix}
\rho \\
m 
\end{pmatrix}, \quad m = \rho u.
\]
Thus its modulated energy, also often refereed to as relative energy, can be defined as
\[
E(U_f|U) := E(U_f) - E(U)  - D E(U)(U_f - U)\quad \mbox{with} \quad U_f := \begin{pmatrix}
        \rho_f \\
         m_f \\
    \end{pmatrix}, \quad m_f = \rho_f u _f.
\]
A straightforward computation gives
\bq\label{energy_mac}
\intr E(U_f|U)\,dx = \frac12 \intr \rho_f |u_f - u|^2\,dx.
\eq
On the other hand, by H\"older inequality, we easily find
\[
\rho_f |u_f|^2 \leq \intr |v|^2 f\,dv.
\]
This yields
$$\begin{aligned}
\intrr f|v-u|^2\,dxdv - \intr \rho_f |u_f-u|^2\,dx
= \intrr |v|^2 f\,dxdv - \intr \rho_f |u_f|^2\,dx \geq 0.
\end{aligned}$$
In fact, we can easily show that
\bq\label{err_energy}
\intrr f|v-u|^2\,dxdv = \intr \rho_f |u_f-u|^2\,dx  + \intrr f|v - u_f|^2\,dxdv.
\eq
This shows that the convergence of the modulated kinetic energy \eqref{energy_mod} implies the convergence of the modulated macro energy \eqref{energy_mac}. We notice that if $f$ is a monokinetic distribution,
$
f(x,v,t) = \rho_f(x,t) \delta_{u_f(x,t)}(v),
$ 
then the second term on the right hand side of \eqref{err_energy} becomes zero, and the two modulated energies \eqref{energy_mod} and \eqref{energy_mac} coincide.  For notational simplicity, we denote by $\mz^N(t) = \{(x_i(t), v_i(t))\}_{i=1}^N$ the set of trajectories associated to the particle system \eqref{main_par}. Then let us define the first important quantity that will allow us to quantify the distance between particles \eqref{main_par} and hydrodynamics \eqref{main_fluid}, it is just the discrete version of the modulated kinetic energy \eqref{energy_mod} defined as
\bq\label{dv_mke}
\me^N (\mz^N(t) | U(t)) :=  \frac12 \intrr |u-v|^2\,\mu^N_t(dxdv) = \frac1{2N} \sum_{i=1}^N |u(x_i(t),t) - v_i(t)|^2.
\eq

The second quantity that will allow us our quantification goal combined with the discrete modulated energy \eqref{dv_mke} is a classical distance between probability measures, the bounded Lipschitz distance, used already by the pioneers in kinetic theory \cite{BH,Neun,Spohn} in the early works for the mean-field limit.  Notice that the pressureless Euler system \eqref{main_fluid} includes the nonlocal position and velocity interaction and alignment forces. Furthermore, its relative energy/entropy has no strict convexity in terms of density variable due to the lack of pressure term. In order to overcome these difficulties, ideas of combining the modulated macro energy and the first or second order Wasserstein distance have been recently proposed in \cite{CCpre, Cpre, CYpre, CCJpre} quantifying the hydrodynamic limit from kinetic equation to the pressureless Euler type system. More recently, in \cite{Cpre2}, a general theory providing some relation between a modulated macro energy-type function and $p$-Wasserstein distance is also developed. In particular, in \cite[Proposition 3.1]{Cpre2}, it is discussed that the $p$-Wasserstein distance with $p\in[1,2]$ can be controlled by the modulated macro energy functional. 

In the present work, we will employ the bounded Lipschitz distance to provide stability estimates between the empirical particle density $\rho^N$ defined as 
\[
\rho^N_t(x) := \int_{\R^d} \mu^N_t\,(dv) = \frac1N \sum_{j=1}^N \delta_{x_j(t)}(x)
\] 
with $\mu^N_t$ be the empirical measure associated to the particle system \eqref{main_par}, and the hydrodynamic particle density $\rho$ solution to \eqref{main_fluid}. More precisely, let $\mathcal{M}(\R^d)$ be the set of nonnegative Radon measures on $\R^d$, which can be considered as nonnegative bounded linear functionals on $\mc_0(\R^d)$. Let $\mu, \nu \in \mathcal{M}(\R^d)$ be two Radon measures. Then the bounded Lipschitz distance, which is denoted by $d_{BL}: \mathcal{M}(\R^d) \times \mathcal{M}(\R^d) \to \R_+$, between $\mu$ and $\nu$ is defined by
\[
d_{BL}(\mu,\nu) := \sup_{\phi \in \Omega} \lt|\intr \phi(x)( \mu(dx) - \nu(dx))\rt|,
\]
where the admissible set $\Omega$ of test functions are given by
\[
\Omega:= \lt\{\phi: \R^d \to \R: \|\phi\|_{L^\infty} \leq 1, \ Lip(\phi) := \sup_{x \neq y} \frac{|\phi(x) - \phi(y)|}{|x-y|} \leq 1 \rt\}.
\]
We also denote by $Lip(\R^d)$ the set of Lipschitz functions on $\R^d$. In Proposition \ref{lem_dbl} below, we provide a relation between the bounded Lispchitz distance and the discrete version of the modulated kinetic energy \eqref{dv_mke}. This key observation allows us to overcome the difficulties mentioned above.


\subsection{Main results and Plan of the paper}

We will first assume that the particle system \eqref{main_par}, the pressureless Euler-type equations \eqref{main_fluid}, and the aggregation equations \eqref{eq_agg}--\eqref{eq_agg2} have existence of smooth enough solutions up to a fixed time $T>0$. We postpone further discussion at the end of this subsection, although we make precise now the assumptions needed on these solutions for our main results. For the limiting hydrodynamic system \eqref{main_fluid}, we need classical solutions in the following sense.

\begin{definition}\label{def_strong} Let $T>0$. We say that $(\rho ,u)$ is a classical solution to the equation \eqref{main_fluid} if the following conditions are satisfied:
\begin{itemize}
\item[(i)] $\rho >0$ on $\R^d \times [0,T)$, $\rho \in \mc([0,T];\mathcal{P}(\R^d))$ and $u \in L^\infty(0,T;\W^{1,\infty}(\R^d))$,
\item[(ii)] $(\rho ,u)$ satisfies \eqref{main_fluid} pointwise. 
\end{itemize}
Here $\mathcal{P}(\R^d)$ denotes the set of probability measures in $\R^d$.
\end{definition}

Our first main result shows the rigorous passage from Newton's equation \eqref{main_par} to pressureless Euler equations \eqref{main_fluid} via the mean-field limit as $N \to \infty$. 

\begin{theorem}\label{thm_main1} Let $T > 0$, $\mz^N(t) = \{(x_i(t), v_i(t))\}_{i=1}^N$ be a solution to the particle system \eqref{main_par}, and let $(\rho, u)$ be the unique classical solution of the pressureless Euler system with nonlocal interaction forces \eqref{main_fluid} in the sense of Definition \ref{def_strong} up to time $T>0$ with initial data $(\rho_0, u_0)$. Then we have
\begin{equation}\label{stab1}\begin{aligned}
\int_{\R^d \times \R^d} |v - u(x,t) |^2 \mu^N_t(dx,dv) &\,+ d^2_{BL}(\rho^N_t(\cdot), \rho(\cdot,t)) \\
&\leq C\left(\int_{\R^d \times \R^d}|v - u_0(x) |^2\mu^N_0(dx,dv) + d^2_{BL}(\rho^N_0, \rho_0)\right) ,
\end{aligned}\end{equation}
where $C>0$ only depends on $\|u\|_{L^\infty \cap Lip}$, $\|\psi\|_{L^\infty \cap Lip}$, $\|\nabla_x W\|_{\W^{1,\infty}}$, and $T$. In particular, if the intial data for \eqref{main_par} and \eqref{main_fluid} are chosen such that
the right hand side of the above inequality goes to zero as $N \to \infty$, then the following consequences hold
$$\begin{aligned}
\intr v \,\mu^N(dv) = \frac1N\sum_{i=1}^N v_i \,\delta_{x_i}  &\rightharpoonup \rho u \quad \mbox{weakly in } L^\infty(0,T^*;\mathcal{M}(\R^d)), \cr
\intr  (v \otimes v)\, \mu^N(dv) = \frac1N\sum_{i=1}^N (v_i \otimes v_i)\,\delta_{x_i}  &\rightharpoonup \rho u \otimes u  \quad \mbox{weakly in } L^\infty(0,T^*;\mathcal{M}(\R^d)), \quad \mbox{and}\cr
\mu^N &\rightharpoonup \rho\delta_{u} \quad \mbox{weakly in } L^\infty(0,T^*;\mathcal{M}(\R^d))
\end{aligned}$$
as $N \to \infty$. 
\end{theorem}

The main novelty of this first result resides in how to control the alignment terms via the modulated energy combined with the bounded Lipschitz distance. 

\begin{remark}[Singular repulsive interaction]
The previous result also applies to singular repulsive interaction potentials. In particular, it holds for the Coulomb interaction potential on $\R^d$ given by
\[
\mathcal{N}(x) = \left\{ \begin{array}{ll}
\displaystyle -\frac{|x|}{2} & \textrm{for $d=1$,}\\[3mm]
\displaystyle -\frac{1}{2\pi} \log |x| & \textrm{for $d=2$,}\\[3mm]
\displaystyle \frac{1}{d(d-2)\alpha_d}\frac{1}{|x|^{d-2}} & \textrm{for $d \geq 3$},
  \end{array} \right. 
\]
and for Riez potentials in a sense to be specified in Subsection 2.3. In fact, the expected stability estimate \eqref{stab1} can be formally substituted by 
$$\begin{aligned}
\int_{\R^d \times \R^d} |v &\,- u(x,t) |^2\mu^N_t(dx,dv) + \intr |\nabla_x \mathcal{N} \star (\rho^N_t - \rho)|^2\,dx + d^2_{BL}(\rho^N_t(\cdot), \rho(\cdot,t)) \cr
&\quad \leq C \left( \int_{\R^d \times \R^d} |v - u_0(x) |^2\mu^N_0(dx,dv) + d^2_{BL}(\rho^N_0, \rho_0) + \intr |\nabla \mathcal{N} \star (\rho^N_0 - \rho_0)|^2\,dx\right),
\end{aligned}$$
where $C>0$ only depends on $\|u\|_{L^\infty \cap Lip}$, $\|\psi\|_{L^\infty \cap Lip}$, and $T$ and corresponding solutions to the particle and the hydrodynamics system for $W=\mathcal{N}$. However, the formal integration by parts leading to such expected stability term for the interaction potential does not make sense due to the singularity of the Newtonian potential. This has been recently solved in the recent breakthrough result in \cite{Spre} by introducing a different relative potential energy avoiding the diagonal terms.
\end{remark}

Section 2 is devoted to the proof of Theorem \ref{thm_main1} and the generalization to singular repulsive potentials using \cite{Spre} in its last subsection.

\

Our second main result is devoted to the asymptotic analysis for the particle system \eqref{main_par} under the small inertia regime: $\e_N \to 0$ as $N \to \infty$. By Theorem \ref{thm_main1}, we expect that for sufficiently large $N\gg 1$, the system \eqref{main_par} in the mean-field/small inertia limit can be well approximated by 
$$\begin{aligned}
&\pa_t \bar\rho + \nabla_x \cdot (\bar\rho \bar u) = 0,\cr
&\e_N\pa_t (\bar\rho  \bar u) + \e_N\nabla_x \cdot (\bar \rho \bar u \otimes \bar u) = -\gamma \bar \rho \bar u - \bar \rho \nabla_x V - \bar \rho \nabla_x W \star\bar\rho + \bar\rho \int_{\R^d} \psi(x-y) (\bar u(y) - \bar u(x))\,\bar\rho(y)\,dy.
\end{aligned}$$
At the formal level, since $\e_N \to 0$ as $N \to \infty$, it follows from the momentum equations in the above system that the hydrodynamic system \eqref{main_fluid} should be replaced by \eqref{eq_agg}--\eqref{eq_agg2} as $N \to \infty$. In order to apply our strategy above, we rewrite the equations \eqref{eq_agg}--\eqref{eq_agg2} as 
\begin{align}\label{main_fluid22}
\begin{aligned}
&\pa_t \bar\rho + \nabla_x \cdot (\bar\rho \bar u) = 0,\cr
&\e_N\pa_t (\bar\rho  \bar u) + \e_N\nabla_x \cdot (\bar \rho \bar u \otimes \bar u) = -\gamma \bar \rho \bar u - \bar \rho \nabla_x V - \bar \rho \nabla_x W \star\bar\rho \cr
&\hspace{4.5cm} + \bar\rho \int_{\R^d} \psi(x-y) (\bar u(y) - \bar u(x))\,\bar\rho(y)\,dy + \e_N \bar\rho \bar e,
\end{aligned}
\end{align}
where $\bar e := \pa_t \bar u + \bar u \cdot \nabla_x \bar u$. We now introduce the needed notion of strong solution to the equation \eqref{eq_agg}--\eqref{eq_agg2} for our purposes.

\begin{definition}\label{def_strong2} Let $T \in (0,\infty)$, we say that $(\bar\rho ,\bar u)$ is a strong solution to the equation \eqref{eq_agg}--\eqref{eq_agg2} if the following conditions are satisfied.
\begin{itemize}
\item[(i)] $\bar\rho \in \mc([0,T];\mathcal{P}(\R^d))$ and $\bar \rho > 0$ on $\R^d \times [0,T)$, 
\item[(ii)] $\bar u \in L^\infty(0,T; \W^{1,\infty}(\R^d))$ and $\pa_t \bar u \in L^\infty(\R^d \times (0,T))$,
\item[(iii)]  $(\bar\rho ,\bar u)$ satisfies \eqref{eq_agg}--\eqref{eq_agg2} pointwise. 
\end{itemize}
\end{definition}

\begin{remark} If $V\equiv 0$ and $\gamma >0$ is sufficiently large, then we can check that $\|\bar u\|_{L^\infty(0,T;\W^{1,\infty})}$ and $\|\pa_t \bar u\|_{L^\infty}$ can be bounded from above by some constant, which depends only on $\|\nabla W\|_{\W^{1,\infty}}$,  $\|\psi\|_{\W^{1,\infty}}$, $\|\bar\rho\|_{L^\infty(0,T;L^1)}$, and $\gamma$. We refer to \cite[Remark 2.5]{Cpre2} for details. For general confinement potentials, we can also deal with general strong solutions for compactly supported initial data since their support remains compact for all times. We refer to \cite{BCLR,CCZ16} for particular instances of these results.
\end{remark}

We can now state our second main result related to a weak-strong stability estimate in the combined mean-field/small inertia limit.

\begin{theorem}\label{thm_main2} Let $T>0$ and $d \geq 1$. Let $\mz^N(t) = \{(x_i(t), v_i(t))\}_{i=1}^N$ be a solution to the particle system \eqref{main_par}, and let $(\bar\rho, \bar u)$ be the unique strong solution of the aggregation-type equation \eqref{eq_agg}--\eqref{eq_agg2} in the sense of Definition \ref{def_strong2} up to time $T>0$ with the initial data $\bar \rho_0$. Suppose that the strength of damping $\gamma >0$ is large enough. Then we have
$$\begin{aligned}
&d_{BL}^2(\rho^N_t(\cdot),\bar \rho(\cdot,t)) +\int_0^t \int_{\R^d \times \R^d} |v - \bar u(x,s) |^2\mu^N_s(dxdv) \,ds \cr
&\quad \leq C\e_N\int_{\R^d \times \R^d} |v - \bar u_0(x) |^2\mu^N_0(dxdv) +Cd_{BL}^2(\rho^N_0,\bar \rho_0)   + C\e_N^2
\end{aligned}$$
and
$$\begin{aligned}
&  \frac1{\e_N}d^2_{BL}(\rho^N_t(\cdot),\bar \rho(\cdot,t))  + \int_{\R^d \times \R^d} |v - \bar u(x,t) |^2\mu^N_t(dxdv)  \cr
&\quad \leq C(1 + \e_N)\int_{\R^d \times \R^d} |v - \bar u_0(x) |^2\mu^N_0(dxdv)   + \frac{C}{\e_N}d_{BL}^2(\rho^N_0,\bar \rho_0)   + C\e_N
\end{aligned}$$
for all $t \in [0,T]$, where $C>0$ is independent of $\e_N$ and $N$ but depending on $\|\bar u\|_{L^\infty(0,T;\W^{1,\infty})}$, $\|\pa_t \bar u\|_{L^\infty}$, $\|\nabla W\|_{\W^{1,\infty}}$,  $\|\psi\|_{\W^{1,\infty}}$, and $\gamma$. In particular if the initial data satisfies
\bq\label{as_thm2}
\int_{\R^d \times \R^d} |v - \bar u_0(x) |^2\mu^N_0(dxdv) + d_{BL}(\rho^N_0,\bar \rho_0) \leq C_0 \,\e_N
\eq
for some $C_0>0$ which is independent of $\e_N$, then we have
\[
d_{BL}^2(\rho^N_t(\cdot),\bar \rho(\cdot,t))  + \int_0^t \int_{\R^d \times \R^d} |v - \bar u(x,s) |^2\mu^N_s(dxdv) \,ds \leq C\e_N^2
\]
and
\[
\int_{\R^d \times \R^d} |v - \bar u(x,t) |^2\mu^N_t(dxdv) \leq C\e_N
\]
for all $t \in [0,T]$, where $C>0$ is as above.
\end{theorem}
\begin{remark}By using the same argument as in Theorem \ref{thm_main1}, under the assumption \eqref{as_thm2} the following consequences hold:
\[
\intr v \,\mu^N(dv) = \frac1N\sum_{i=1}^N v_i \,\delta_{x_i}  \rightharpoonup \bar \rho \bar u \quad \mbox{weakly in } L^\infty(0,T^*;\mathcal{M}(\R^d))
\]
and
\[
\mu^N \rightharpoonup \bar\rho\delta_{\bar u} \quad \mbox{weakly in } L^\infty(0,T^*;\mathcal{M}(\R^d))
\]
as $N \to \infty$. 
\end{remark}

Section 3 is devoted to the proof of Theorem \ref{thm_main2} and the generalizations to  singular repulsive potentials. Finally, we complement these results by showing the existence of solutions to the particle system \eqref{main_par} in Appendix \ref{app_par} and the existence and uniqueness of strong solutions in the sense of Definition \ref{def_strong} for the hydrodynamic system \eqref{main_fluid} in the final Section 4 of this paper.

%
%
%
%
\section{Mean-field limit: from Newton to pressureless Euler}

In this section, we provide the details of the proof for Theorem \ref{thm_main1}. As mentioned before, one of our main mathematical tools is the discrete version of the modulated kinetic energy $\me^N (\mz^N(t) | U(t))$ defined in \eqref{dv_mke}.

\subsection{Proof of Theorem \ref{thm_main1}: quantitative bound estimate}

In this part, our main purpose is to give the quantitative bound estimate of the discrete modulated kinetic energy $\me^N (\mz^N(t) | U(t))$. 

\begin{proposition}\label{lem_mke}  Let $T > 0$, $\mz^N(t) = \{(x_i(t), v_i(t))\}_{i=1}^N$ be a solution to the particle system \eqref{main_par}, and let $(\rho, u)$ be the unique classical solution of the pressureless Euler system with nonlocal interaction forces \eqref{main_fluid} in the sense of Definition \ref{def_strong} up to time $T>0$. Then we have
\begin{align}\label{mke_ineq}
\begin{aligned}
&\frac{d}{dt} \me^N (\mz^N(t) | U(t))+ 2\gamma \me^N (\mz^N(t) | U(t)) + \frac1N \sum_{i=1}^N \psi(x_i - y_i)|v_i - u(x_i)|^2\cr
&\quad \leq C\me^N (\mz^N(t) | U(t)) + C d_{BL}^2(\rho^N_t(\cdot),\rho(\cdot,t)),
\end{aligned}
\end{align}
where $C>0$ is independent of $N$ and $\gamma$.
\end{proposition}

\begin{proof}
By the notion of our classical solution, we obtain from the momentum equation in \eqref{main_fluid} that
$$\begin{aligned}
\pa_t (u(x_i(t),t)) &=  v_i(t) \cdot \nabla_x u(x_i(t),t)  + (\pa_t u)(x_i(t),t)\cr
&= (v_i(t) - u(x_i(t),t)) \cdot  \nabla_x u(x_i(t),t) - \gamma u(x_i(t)) - \nabla_x V(x_i(t)) - (\nabla_x W \star \rho)(x_i) \cr
&\hspace{2cm} + \int_{\R^d} \psi(x_i(t) - y) (u(y,t) - u(x_i(t),t) )\rho(y,t)\,dy.
\end{aligned}$$
Then using this and \eqref{main_par}, we estimate the discrete modulated kinetic energy functional as
\begin{align}\label{est_mod}
\begin{aligned}
\frac{d}{dt} \me^N (\mz^N(t) | U(t)) = &\, \frac1N \sum_{i=1}^N (u(x_i(t),t) - v_i(t)) \cdot \lt( \pa_t u(x_i(t),t) + v_i(t) \cdot \nabla_x u(x_i(t),t)  - \dot v_i(t)\rt)\cr
=&\,\frac1N \sum_{i=1}^N (u(x_i(t),t) - v_i(t)) \cdot ((v_i(t) - u(x_i(t),t))\cdot \nabla_x)u(x_i(t),t) \cr
& - \frac\gamma N \sum_{i=1}^N |u(x_i(t),t) - v_i(t)|^2 \cr
&  - \frac1N \sum_{i=1}^N (u(x_i(t),t) - v_i(t))  \cdot  \lt((\nabla_x W \star \rho)(x_i) - (\nabla_x W \star \rho^N)(x_i)\rt) \cr
& + \frac1N \sum_{i=1}^N (u(x_i(t),t) - v_i(t)) \cdot F(x_i(t), v_i(t))\cr
=: &\,\sum_{i=1}^4 I_i,
\end{aligned}
\end{align}
where 
$$\begin{aligned}
F(x_i(t), v_i(t))&:= \int_{\R^d} \psi(x_i(t) - y) (u(y,t) - u(x_i(t),t) )\rho(y,t)\,dy - \frac1N \sum_{j=1}^N \psi(x_i(t) - x_j(t))(v_j(t) - v_i(t)).
\end{aligned}$$
Here $I_1$ can be easily estimated as
$$\begin{aligned}
I_1 &= \frac1N \sum_{i=1}^N \nabla_x u(x_i(t),t) : (u(x_i(t),t) - v_i(t)) \otimes (v_i(t) - u(x_i(t),t))\cr
&\leq \|\nabla_x u(\cdot,t)\|_{L^\infty}\frac1N \sum_{i=1}^N |u(x_i(t),t) - v_i(t)|^2 \cr
&= 2 \|\nabla_x u(\cdot,t)\|_{L^\infty}\me^N (\mz^N(t) | U(t)).
\end{aligned}$$
By definition, we obtain
$
I_2 = -2\gamma \me^N (\mz^N(t) | U(t)).
$
We next estimate $I_3$ as
$$\begin{aligned}
I_3 &= - \frac1N \sum_{i=1}^N (u(x_i(t),t) - v_i(t)) \cdot  \lt((\nabla_x W \star \rho)(x_i(t),t) - (\nabla_x W \star \rho^N)(x_i(t),t) \rt)\cr
&= \frac1N \sum_{i=1}^N (v_i(t) - u(x_i(t),t)) \cdot (\nabla_x W \star (\rho - \rho^N))(x_i(t),t).
\end{aligned}$$
On the other hand, the fact $\nabla_x W \in \W^{1,\infty}$ gives
\[
\|(\nabla_x W \star (\rho - \rho^N))(\cdot,t)\|_{L^\infty} \leq \|\nabla_x W\|_{\W^{1,\infty}}d_{BL}(\rho^N,\rho),
\]
and subsequently this asserts
$$\begin{aligned}
I_3 &\leq \|\nabla_x W\|_{\W^{1,\infty}}d_{BL}(\rho^N,\rho)\lt( \frac1N \sum_{i=1}^N |v_i(t) - u(x_i(t),t)| \rt) \cr
&\leq \|\nabla_x W\|_{\W^{1,\infty}}d_{BL}(\rho^N,\rho)\lt( \frac1N \sum_{i=1}^N |v_i(t) - u(x_i(t),t)|^2 \rt)^{1/2} \cr
&=\|\nabla_x W\|_{\W^{1,\infty}}d_{BL}(\rho^N,\rho) \sqrt{\me^N (\mz^N(t) | U(t))}.
\end{aligned}$$
For the estimate of $I_4$, we note that
$$\begin{aligned}
&\frac1N \sum_{j=1}^N \psi(x_i(t) - x_j(t))(v_j(t) - v_i(t))\cr
&\quad = \frac1N \sum_{j=1}^N \psi(x_i(t) - x_j(t))(v_j(t) - u(x_j(t),t)) +  \frac1N \sum_{j=1}^N \psi(x_i(t) - x_j(t))(u(x_j(t),t) - v_i(t))\cr
&\quad =: J_1 + J_2.
\end{aligned}$$
Then we rewrite $J_2$ as 
\[
J_2 = \int_{\R^d} \psi(x_i(t) - y) (u(y,t) - v_i(t))\rho^N(y,t)\,dy.
\]
This yields
$$\begin{aligned}
I_4 &= \frac1N \sum_{i=1}^N (u(x_i) - v_i) \cdot \frac1N \sum_{j=1}^N \psi(x_i - x_j)(u(x_j) - v_j)\cr
&\quad + \frac1N \sum_{i=1}^N (u(x_i) - v_i) \cdot \lt(\int_{\R^d} \psi(x_i - y) (u(y) - u(x_i) )\rho(y)\,dy -  \int_{\R^d} \psi(x_i - y) (u(y) - v_i)\rho^N(y)\,dy \rt)\cr
&=: I_4^1 + I_4^2.
\end{aligned}$$
Here we can easily estimate $I_4^1$ as 
\[
I_4^1 \leq \|\psi\|_{L^\infty}\lt(\frac1N \sum_{i=1}^N (u(x_i) - v_i)  \rt)^2 \leq \|\psi\|_{L^\infty}\frac1N \sum_{i=1}^N |u(x_i) - v_i|^2 = 2 \|\psi\|_{L^\infty}\me^N (\mz^N(t) | U(t)).
\]
Note that 
$$\begin{aligned}
&\frac1N\sum_{i=1}^N \int_{\R^d} \psi(x_i - y)(v_i - u(x_i)) (\rho^N(y) - \rho(y)) \cdot (u(y) - u(x_i))\,dy\cr
&\quad = \frac1N\sum_{i=1}^N \int_{\R^d} \psi(x_i - y)(v_i - u(x_i)) \rho^N(y) \cdot (u(y) - u(x_i))\,dy\cr
&\qquad + I_4^2 - \frac1N\sum_{i=1}^N \int_{\R^d} \psi(x_i - y)(v_i - u(x_i)) \rho^N(y) \cdot (u(y) - v_i)\,dy\cr
&\quad = I_4^2 + \frac1N \sum_{i=1}^N \psi(x_i - y_i)|v_i - u(x_i)|^2, 
\end{aligned}$$
that is,
$$\begin{aligned}
I_4^2  &= \frac1N\sum_{i=1}^N \int_{\R^d} \psi(x_i - y)(v_i - u(x_i)) (\rho^N(y) - \rho(y)) \cdot (u(y) - u(x_i))\,dy - \frac1N \sum_{i=1}^N \psi(x_i - y_i)|v_i - u(x_i)|^2.
\end{aligned}$$
On the other hand, we can estimate
$$\begin{aligned}
&\frac1N\sum_{i=1}^N \int_{\R^d} \psi(x_i - y)(v_i - u(x_i)) (\rho^N(y) - \rho(y)) \cdot (u(y) - u(x_i))\,dy\cr
&\quad = \frac1N\sum_{i=1}^N (v_i - u(x_i)) \cdot \int_{\R^d} \psi(x_i - y)u(y) (\rho^N(y) - \rho(y))  \,dy\cr
&\qquad  - \frac1N\sum_{i=1}^N(v_i - u(x_i)) \cdot u(x_i) \int_{\R^d} \psi(x_i - y) (\rho^N(y) - \rho(y)) \,dy\cr
&\quad =: K_1 + K_2,
\end{aligned}$$
where
$$\begin{aligned}
\lt|K_1\rt| &\leq \frac1N\sum_{i=1}^N |v_i - u(x_i)| \lt|\int_{\R^d} \psi(x_i - y)u(y) (\rho^N(y) - \rho(y))  \,dy\rt|\cr
&\leq \|\psi u\|_{L^\infty \cap Lip}\frac1N\sum_{i=1}^N |v_i - u(x_i)| \,d_{BL}(\rho^N,\rho)\cr
&\leq \|\psi u\|_{L^\infty \cap Lip}\lt( \frac1N\sum_{i=1}^N |v_i - u(x_i)|^2\rt)^{1/2}d_{BL}(\rho^N,\rho)\cr
&\leq \|\psi u\|_{L^\infty \cap Lip}\sqrt 2 \sqrt{\me^N (\mz^N(t) | U(t))} \,d_{BL}(\rho^N,\rho).
\end{aligned}$$
Similarly, we also find
$$\begin{aligned}
\lt|K_2\rt| &\leq \frac1N\sum_{i=1}^N |v_i - u(x_i)| |u(x_i)|\lt|\int_{\R^d} \psi(x_i - y) (\rho^N(y) - \rho(y))  \,dy\rt|\cr
&\leq \|u\|_{L^\infty}\|\psi\|_{L^\infty \cap Lip}\sqrt 2 \sqrt{\me^N (\mz^N(t) | U(t))} \,d_{BL}(\rho^N,\rho).
\end{aligned}$$
Combining all of the above estimates, we have
$$\begin{aligned}
&\frac{d}{dt} \me^N (\mz^N(t) | U(t))+ 2\gamma \me^N (\mz^N(t) | U(t)) + \frac1N \sum_{i=1}^N \psi(x_i - y_i)|v_i - u(x_i)|^2\cr
&\quad \leq 2\lt(\|\nabla_x u(\cdot,t)\|_{L^\infty} + \|\psi\|_{L^\infty}\rt)\me^N (\mz^N(t) | U(t)) \cr
&\qquad + \sqrt 2\lt( \|\psi u\|_{L^\infty \cap Lip}+ \|u(\cdot,t)\|_{L^\infty}\|\psi\|_{L^\infty \cap Lip} + \|\nabla_x W\|_{\W^{1,\infty}}\rt) \sqrt{\me^N (\mz^N(t) | U(t))} \,d_{BL}(\rho^N_t(\cdot),\rho(\cdot,t)).
\end{aligned}$$
This completes the proof.
\end{proof}

\begin{remark} We assumed that the communication weight $\psi$ is nonnegative, which takes into account the velocity alignment forces, however a similar bound estimate for the discrete kinetic energy $\me^N$ to that in Proposition \ref{lem_mke} can be obtained. Indeed, if $\psi$ can be negative, but bounded, then the third term on the left hand side of \eqref{mke_ineq} can be estimated as
\[
\lt|\frac1N \sum_{i=1}^N \psi(x_i - y_i)|v_i - u(x_i)|^2\rt| \leq 2\|\psi\|_{L^\infty} \me^N (\mz^N | U).
\]
This yields
\[
\frac{d}{dt} \me^N (\mz^N(t) | U(t))+ 2\gamma\me^N (\mz^N(t) | U(t))\leq C\me^N (\mz^N(t) | U(t)) + C d_{BL}^2(\rho^N_t(\cdot),\rho(\cdot,t)),
\]
where $C>0$ is independent of $N$ and $\gamma$.
\end{remark}

In order to close the estimate in Proposition \ref{lem_mke}, we need to estimate the bounded Lipschitz distance between $\rho^N$ and $\rho$. In the proposition below, we provide the relation between the bounded Lipschitz distance and the discrete modulated kinetic energy.

\begin{proposition}\label{lem_dbl} Let $\rho^N$ and $\rho$ be defined as above. Then we have
\[
d_{BL}^2(\rho^N(\cdot,t),\rho(\cdot,t)) \leq Cd_{BL}^2(\rho^N_0,\rho_0) + C\int_0^t \me^N (\mz^N(s) | U(s))\,ds,
\]
where $C > 0$ depends only on $\|u\|_{L^\infty(0,T;Lip)}$ and $T$. 
\end{proposition}

\begin{proof}
Consider a forward characteristics $\eta=\eta(x,t)$ for the system \eqref{main_fluid} satisfying the following ODEs:
\bq\label{eq_char}
\frac{d\eta(x,t)}{dt} = u(\eta(x,t),t)
\eq
subject to the initial data:
$
\eta(x,0) = x \in \R^d.
$
The characteristic $\eta$ is well-defined because of the Lipschitz continuous regularity of $u$. Note that along the characteristic, the solution $\rho$ can be written as the mild form:
\[
\rho(\eta(x,t),t) = \rho_0(x) \exp\lt(-\int_0^t (\nabla_x \cdot u)(\eta(x,s),s)\,ds\rt),
\] 
and thus we get
\[
\rho_0(x) = \rho(\eta(x,t),t)\exp\lt(\int_0^t (\nabla_x \cdot u)(\eta(x,s),s)\,ds\rt) = \rho(\eta(x,t),t) det\lt((\nabla_x \eta)(x,t) \rt).
\]
This together with using the change of variables yields 
\bq\label{push_rho}
\intr \phi (\eta(x,t)) \rho_0(x)\,dx = \intr \phi (\eta(x,t)) \rho(\eta(x,t),t) det\lt((\nabla_x \eta)(x,t) \rt)dx = \intr \phi(x) \rho(x,t)\,dx
\eq
for $\phi \in (L^\infty \cap Lip)(\R^d)$. Moreover, we find from \eqref{eq_char} that
\begin{align}\label{lip_eta}
\begin{aligned}
\lt|\eta(x,t) - \eta(y,t)\rt| &= \lt| x- y + \int_0^t \lt(u(\eta(x,s),s) - u(\eta(y,s),s)\rt) ds \rt|\cr
&\leq |x-y| + \|u\|_{Lip} \int_0^t |\eta(x,s) - \eta(y,s)|\,ds,
\end{aligned}
\end{align}
and applying Gr\"onwall's lemma to the above gives
\[
\lt|\eta(x,t) - \eta(y,t)\rt| \leq C|x - y|,
\]
where $C>0$ depends only on $\|u\|_{L^\infty(0,T;Lip)}$ and $T$, i.e., $\eta$ is Lipschitz continuous in $\R^d$. We also get 
\[
|x_i(t) - \eta(x,t)| \leq |x_i(0) - x| + \int_0^t | v_i(s) - u(\eta(x,s),s)|\,ds.
\]
Here the second term on the right hand side of the above inequality can be estimated as
$$\begin{aligned}
\int_0^t | v_i(s) - u(\eta(x,s),s)|\,ds
&\leq \int_0^t | v_i(s) - u(x_i(s),s)|\,ds + \int_0^t | u(x_i(s),s) - u(\eta(x,s),s)|\,ds\cr
& \leq \int_0^t | v_i(s) - u(x_i(s),s)|\,ds + \|u\|_{Lip} \int_0^t | x_i(s) - \eta(x,s)|\,ds.
\end{aligned}$$
Thus we get
$$\begin{aligned}
|x_i(t) - \eta(x,t)| &\leq |x_i(0) - x| +\int_0^t | v_i(s) - u(x_i(s),s)|\,ds  + \|u\|_{Lip} \int_0^t | x_i(s) - \eta(x,s)|\,ds,
\end{aligned}$$
and applying Gr\"onwall's lemma to the above deduces
\[
|x_i(t) - \eta(x,t)| \leq C|x_i(0) - x| +C\int_0^t | v_i(s) - u(x_i(s),s)|\,ds,
\]
where $C$ depends only on $\|u\|_{L^\infty(0,T;Lip)}$ and $T$. In particular, by taking $x = x_i(0)$, we get
\bq\label{est_d1}
|x_i(t) - \eta(x_i(0),t)| \leq C\int_0^t | v_i(s) - u(x_i(s),s)|\,ds.
\eq
Then for any $\phi \in (L^\infty \cap Lip)(\R^d)$ we use \eqref{push_rho} to estimate 
\begin{align}\label{est_l}
\begin{aligned}
\lt|\int_{\R^d} \!\!\phi(x) (\rho^N - \rho)\,dx\rt| 
& \!=\! \lt|\frac1N \sum_{i=1}^N \phi(x_i(t)) \!-\! \intr \phi(\eta(x,t))\rho_0\,dx \rt|\cr
&= \!\lt|\frac1N \sum_{i=1}^N (\phi(x_i(t)) \!- \!\phi(\eta(x_i(0),t))) \!+\! \frac1N\sum_{i=1}^N \phi(\eta(x_i(0),t))\!-\! \intr \!\!\phi(\eta(x,t))\rho_0\,dx \rt|\cr
& \leq\! \frac1N \sum_{i=1}^N |\phi(x_i(t)) - \phi(\eta(x_i(0),t))|\! +\! \lt|\frac1N\sum_{i=1}^N \phi(\eta(x_i(0),t)) \!- \!\intr \!\!\phi(\eta(x,t))\rho_0\,dx \rt| \cr
&=: L_1 + L_2.
\end{aligned}
\end{align}
For $L_1$, we use the Lipschitz continuity together with \eqref{est_d1} to obtain
\begin{align}\label{est_l1}
\begin{aligned}
L_1 &\leq \frac{\|\phi\|_{Lip}}N\sum_{i=1}^N |x_i(t) - \eta(x_i(0),t)| \leq \frac{\|\phi\|_{Lip}}N\int_0^t\sum_{i=1}^N| v_i(s) - u(x_i(s),s)|\,ds\cr
&\leq \|\phi\|_{Lip}\sqrt{T}\lt( \int_0^t \frac1N\sum_{i=1}^N| v_i(s) - u(x_i(s),s)|^2\,ds\rt)^{1/2}\!\!\!\!\!\!= \|\phi\|_{Lip}\sqrt{T}\lt(\int_0^t \me^N (\mz^N(s) | U(s))\,ds \rt)^{1/2}.
\end{aligned}
\end{align}
For the estimate of $L_2$, we notice that 
\[
\frac1N\sum_{i=1}^N \phi(\eta(x_i(0),t)) = \intr \phi(\eta(x,t))\rho^N_0\,dx.
\]
Using this identity, the Lipschitz estimate for $\eta$ in \eqref{lip_eta}, and the fact $\phi \in (L^\infty \cap Lip)(\R^d)$, we find
\bq\label{est_l2}
L_2 = \lt| \intr \phi(\eta(x,t))(\rho^N_0 - \rho_0)\,dx \rt| \leq \lt(\|\phi\|_{L^\infty} + \|\phi\|_{Lip}\|\eta\|_{Lip}\rt) \,d_{BL}(\rho^N_0, \rho_0).
\eq
Putting \eqref{est_l1} and \eqref{est_l2} into \eqref{est_l} yields
\[
d_{BL}(\rho^N_t(\cdot),\rho(\cdot,t))\leq Cd_{BL}(\rho^N_0, \rho_0) + C\lt(\int_0^t \me^N (\mz^N(s) | U(s))\,ds \rt)^{1/2}
\]
for $0 \leq t \leq T$, where $C > 0$ depends only on $\|u\|_{L^\infty(0,T;Lip)}$ and $T$.
\end{proof}

\begin{proof}[Proof of Theorem \ref{thm_main1}] Applying Gr\"onwall's lemma and Young's inequality to the differential inequality in Proposition \ref{lem_mke} yields
\[
\me^N (\mz^N(t) | U(t))  \leq C\me^N (\mz^N_0 | U_0)  + C \int_0^t d^2_{BL}(\rho^N_s(\cdot),\rho(\cdot,s))\,ds,
\]
where $C > 0$ is independent of $N$. We then use Proposition \ref{lem_dbl} to have
$$\begin{aligned}
\me^N (\mz^N(t) | U(t))  + d_{BL}^2(\rho^N_t(\cdot),\rho(\cdot,t))
\leq &\, C\me^N (\mz^N_0 | U_0)   + Cd_{BL}^2(\rho^N_0,\rho_0) \cr
&+ C \int_0^t d^2_{BL}(\rho^N_s(\cdot),\rho(\cdot,s))\,ds + C\int_0^t \me^N (\mz^N(s) | U(s))\,ds.
\end{aligned}$$
We finally apply Gr\"onwall's to the above to conclude the desired result.
\end{proof}

\subsection{Proof of Theorem \ref{thm_main1}: convergence estimates} In this part, we provide the details on the proof for convergences appeared in Theorem \ref{thm_main1}. For this, it suffices to prove the following lemma.

\begin{lemma}
\begin{itemize}
\item[(i)] Convergence of local moment:
\[
d_{BL}\lt(\intr v \,\mu^N(dv), \, \rho u\rt) \leq \lt(\intrr |v - u(x) |^2\mu^N(dxdv) \rt)^{1/2} + Cd_{BL}(\rho^N, \rho).
\]
\item[(ii)] Convergence of local energy:
$$\begin{aligned}
&d_{BL}\lt(\intr  (v \otimes v)\, \mu^N(dv), \,\rho u \otimes u\rt) \cr
&\quad \leq \intrr |v - u(x) |^2\mu^N(dxdv) + C\lt(\intrr |v - u(x) |^2\mu^N(dxdv) \rt)^{1/2}+ Cd_{BL}(\rho^N, \rho).
\end{aligned}$$
\item[(iii)] Convergence of empirical measure:
\[
d^2_{BL}(\mu^N, \rho\delta_{u}) \leq C\int_{\R^d \times \R^d} |v - u(x) |^2\,\mu^N(dxdv) + Cd^2_{BL}(\rho^N, \rho).
\]
\end{itemize}
Here $C>0$ is independent of $N$.
\end{lemma}
\begin{proof}
(i)  For any $\phi \in (L^\infty \cap Lip)(\R^d)$, we get
$$\begin{aligned}
&\lt|\intr \phi(x) \lt(\intr v\, \mu^N(x,dv) - (\rho u)(x) \rt)dx \rt|\cr
&\quad = \lt|\intrr \phi(x) (v-u(x))\,\mu^N(dxdv) + \intr \phi(x) u(x) (\rho^N(x) - \rho(x))\,dx \rt|\cr
&\quad \leq \|\phi\|_{L^\infty}\lt(\intrr  |v-u(x)|\,\mu^N(dxdv) \rt) + \|\phi u\|_{L^\infty \cap Lip} \,d_{BL}(\rho^N, \rho)\cr
&\quad \leq  \|\phi\|_{L^\infty}\lt(\intrr  |v-u(x)|^2\,\mu^N(dxdv) \rt)^{1/2} \cr
&\qquad + \lt(\|\phi\|_{L^\infty} \|u\|_{L^\infty} + \|\phi\|_{L^\infty}\|u\|_{Lip} + \|u\|_{L^\infty}\|\phi\|_{Lip} \rt)d_{BL}(\rho^N, \rho).
\end{aligned}$$

(ii) Adding and subtracting, we notice that
$$\begin{aligned}
\intr (v \otimes v)\, \mu^N(dv) - \rho u\otimes u&= \intr (v - u) \otimes (v-u)\, \mu^N(dv) + u \otimes \lt(\intr v\mu^N(dv) - \rho u \rt) \cr
&\quad + \lt(\intr v\mu^N(dv) - \rho u \rt)\otimes u + (\rho - \rho^N) u\otimes u.
\end{aligned}$$
This yields for $\phi \in (L^\infty \cap Lip)(\R^d)$
$$\begin{aligned}
&\lt|\intr \phi(x)\lt(\intr (v \otimes v)\, \mu^N(dv) - (\rho u)(x)\otimes u(x)\rt)dx \rt|\cr
&\quad \leq \|\phi\|_{L^\infty} \intrr |v - u|^2 \, \mu^N(dxdv) + 2\|\phi u\|_{L^\infty\cap Lip}\,d_{BL}\lt(\intr v \,\mu^N(dv), \ \rho u\rt) \cr
&\qquad + \|\phi |u|^2\|_{L^\infty \cap Lip} \,d_{BL}(\rho^N, \rho).
\end{aligned}$$

(iii)  For any $\varphi \in (L^\infty \cap Lip)(\R^d \times \R^d)$, we find
$$\begin{aligned}
&\lt|\intrr  \varphi(x,v) \lt(\mu^N(dxdv) - \rho(x)dx \otimes\delta_{u(x)}(dv)\rt)\rt|\cr
&\quad = \lt| \intrr  \varphi(x,v)\, \mu^N(dxdv) - \intr \varphi(x,u(x))\rho(x)\,dx \rt|\cr
&\quad = \lt| \intrr  (\varphi(x,v) -\varphi(x,u(x)))\,\mu^N(dxdv)+ \intr \varphi(x,u(x))(\rho^N - \rho)(x)\,dx\rt|\cr
&\quad \leq \|\varphi\|_{Lip}\intrr  |v - u(x)|\,\mu^N(dxdv) + (\|\varphi\|_{L^\infty} + \|\varphi\|_{Lip}\|u\|_{Lip})d_{BL}(\rho^N, \rho)\cr
&\quad \leq C\lt(\intrr  |v - u(x)|^2\,\mu^N(dxdv) \rt)^{1/2} + Cd_{BL}(\rho^N, \rho).
\end{aligned}$$
\end{proof}

\subsection{Singular interaction potential cases: Coulomb and Riesz potentials }
In this part, we discuss the singular interaction potentials. Let $d \geq 1$ and consider a potential $\wt W$ has the form of
\bq\label{w_a}
\wt W(x) = |x|^{-\alpha} \quad \max\{d-2,0\} \leq \alpha < d \quad \forall \, d \geq 1 
\eq
or 
\bq\label{w_a2}
\wt W(x) = -\log|x| \quad \mbox{for $d=1$ or $2$}.
\eq
Note that the case $\alpha=d-2$ with $d \geq 3$ or \eqref{w_a2} with $d=2$ corresponds to the Coulomb potential, and the other cases are called 
Riesz potentials. With these types of singular potentials, in a recent work \cite{Spre}, the quantitative mean-field limit from the particle system \eqref{main_par} to the pressureless Euler-type system when $\gamma = 0$, $V\equiv 0$ and $\psi \equiv 0$. More precisely, in \cite{Spre}, the following modulated free energy is employed to measure the error between particle and continuum systems:
\[
\mf^N (\mz^N(t) | U(t)):= \frac12\int_{\R^d \times \R^d \setminus \Delta} \wt W(x-y)(\rho^N - \rho)(x)(\rho^N - \rho)(y)\,dxdy,
\]
where $\Delta$ denotes the diagonal in $\R^d \times \R^d$.

\begin{theorem}\label{thm_main12} Let $T > 0$ and $\mz^N(t) = \{(x_i(t), v_i(t))\}_{i=1}^N$ be a solution to the particle system \eqref{main_par}, and let $(\rho, u)$ be the unique classical solution of the pressureless Euler system \eqref{main_fluid} with nonlocal interaction forces $\wt W$, which is appeared in \eqref{w_a} or \eqref{w_a2}, instead of $W$ up to time $T>0$ with initial data $(\rho_0, u_0)$. Assume that the classical solution $(\rho,u)$ satisfies $\rho \in L^\infty(0,T;(\mathcal{P} \cap L^\infty)(\R^d))$ and $u \in L^\infty(0,T;\W^{1,\infty}(\R^d))$. In the case $s \geq d-1$, we further assume that $\rho \in L^\infty(0,T;\mc^\sigma(\R^d))$ for some $\sigma > \alpha-d+1$. Then there exists $\beta < 2$ such that 
\begin{align}\label{res_thm12}
\begin{aligned}
\int_{\R^d \times \R^d} |v - u(x,t) |^2\,\mu^N_t(dxdv) &\,+ d^2_{BL}(\rho^N_t(\cdot), \rho(\cdot,t))  
+ \int_{\R^d \times \R^d \setminus \Delta} \wt W(x-y)(\rho^N - \rho)(x)(\rho^N - \rho)(y)\,dxdy  \cr
&\leq C\int_{\R^d \times \R^d} |v - u_0(x) |^2\,\mu^N_0(dxdv) + Cd^2_{BL}(\rho^N_0, \rho_0) \cr
&\quad + C \int_{\R^d \times \R^d \setminus \Delta} \wt W(x-y)(\rho^N_0 - \rho_0)(x)(\rho^N_0 - \rho_0)(y)\,dxdy + CN^{\beta-2},
\end{aligned}
\end{align}
where $C>0$ is independent of $N$. 
\end{theorem}
\begin{remark}If the interaction potential $W$ is singular at the origin, then the term related to $W$ in \eqref{main_par} should be replaced by $\frac1N \sum_{j:j \neq i} \nabla_x W (x_i - x_j)$ since $W(0)$ can not be well defined. This is why the diagonal $\Delta$ is excluded in the integration in the modulated potential energy. 
\end{remark}

\begin{remark}If the right hand side of \eqref{res_thm12} converges to zero as $N \to \infty$, then we also have the same convergence estimates in Theorem \ref{thm_main1}.
\end{remark}
\begin{proof}[Proof of Theorem \ref{thm_main12}] For the proof, we only need to reestimate $I_3$ term in the proof of Proposition \ref{lem_mke}. Although this proof is almost the same with that of \cite{Spre}, we provide the details here for the completeness of our work. Let us denote by 
$$\begin{aligned}
I &:= - \frac1N \sum_{i=1}^N \int_{\R^d} (u(x_i(t),t) - v_i(t)) \cdot \nabla_x \wt W(x_i(t) - y) \rho(y,t)\,dy\cr
&\quad  + \frac{1}{N^2} \sum_{i \neq j} (u(x_i(t),t) - v_i(t)) \cdot  \nabla_x \wt W(x_i(t) - x_j(t)).
\end{aligned}$$
On the other hand, we find
$$\begin{aligned}
\frac{d}{dt}\mf^N (\mz^N(t) | U(t))
= &\,\frac12\frac{d}{dt}\lt(\frac{1}{N^2} \sum_{i \neq j} \wt W(x_i - x_j) \rt) - \frac{d}{dt}\lt(\frac1N \sum_{i=1}^N \int_{\R^d} \wt W(x_i - y) \rho(y)\,dy \rt)\cr
&+ \frac12\frac{d}{dt}\lt(\int_{\R^d \times \R^d}\wt W(x - y)\rho(x)\rho(y)\,dxdy \rt) \cr
= &\, \frac{1}{N^2}  \sum_{i \neq j} \nabla_x \wt W(x_i - x_j) \cdot v_i - \frac1N \sum_{i=1}^N \int_{\R^d}\nabla_x \wt W(x_i - y) \cdot v_i \rho(y)\,dy\cr
& - \frac1N\sum_{i=1}^N \int_{\R^d}\nabla_x \wt W(x_i - y) \cdot (\rho u)(y)\,dy + \int_{\R^d \times \R^d} \nabla_x \wt W(x-y) (\rho u)(x) \rho(y)\,dxdy.
\end{aligned}$$
Here we used 
\bq\label{asym}
\nabla_x W(-x) = -\nabla_x W(x) \quad \mbox{for} \quad x \in \R^d \setminus \{0\}.
\eq This implies
$$\begin{aligned}
I &:= - \frac12\frac{d}{dt} \int_{\R^d \times \R^d \setminus \Delta} \wt W(x-y)(\rho^N - \rho)(x)(\rho^N - \rho)(y)\,dxdy\cr
&\quad + \frac1{N^2} \sum_{i\neq j} u(x_i) \cdot \nabla_x \wt W(x_i - x_j) - \frac1N\sum_{i=1}^N \int_{\R^d} \nabla_x \wt W(x_i -y) \cdot (u(x_i) - u(y)) \rho(y)\,dy\cr
&\quad + \int_{\R^d \times \R^d} \nabla_x \wt W(x-y) (\rho u)(x) \rho(y)\,dxdy.
\end{aligned}$$
We next use \eqref{asym} to get 
\[
\frac1{N^2} \sum_{i\neq j} u(x_i) \cdot \nabla_x \wt W(x_i - x_j) = \frac12\frac1{N^2} \sum_{i\neq j} \lt(u(x_i) - u(x_j)\rt) \cdot \nabla_x \wt W(x_i - x_j)
\]
and
\[
\int_{\R^d \times \R^d} \nabla_x \wt W(x-y) (\rho u)(x) \rho(y)\,dxdy = \frac12\int_{\R^d \times \R^d} \nabla_x \wt W(x-y) \lt( u(x) - u(y)\rt) \rho(x) \rho(y)\,dxdy.
\]
Thus we obtain
$$\begin{aligned}
I &:= - \frac12\frac{d}{dt} \int_{\R^d \times \R^d \setminus \Delta} \wt W(x-y)(\rho^N - \rho)(x)(\rho^N - \rho)(y)\,dxdy\cr
&\quad\,\, +  \frac12\int_{\R^d \times \R^d \setminus \Delta} \lt(u(x) - u(y) \rt) \cdot \nabla_x \wt W (x-y) (\rho^N - \rho)(x) (\rho^N - \rho)(y)\,dxdy.
\end{aligned}$$
This together with the estimates in Proposition \ref{lem_mke} yields
$$\begin{aligned}
\frac{d}{dt}\lt( \me^N (\mz^N(t) | U(t)) \rt. \!\!&\lt.+\mf^N (\mz^N(t) | U(t))\rt)  + 2\gamma \me^N (\mz^N(t) | U(t)) + \frac1N \sum_{i=1}^N \psi(x_i - y_i)|v_i - u(x_i)|^2\cr
&\leq C\me^N (\mz^N(t) | U(t))  + C d_{BL}^2(\rho^N,\rho)\cr
& \quad + \frac12\int_{\R^d \times \R^d \setminus \Delta} \lt(u(x) - u(y) \rt) \cdot \nabla_x \wt W (x-y) (\rho^N - \rho)(x) (\rho^N - \rho)(y)\,dxdy.
\end{aligned}$$
We then apply \cite[Proposition 1.1]{Spre} to have that the last term on the right hand side of the above inequality can be bounded from above by
\[
C\mf^N (\mz^N(t) | U(t)) + CN^{\beta-2}
\]
for some $\beta < 2$, where $C>0$ is independent of $N$. Applying the Gr\"onwall's lemma to the resulting inequality concludes the desired result.
\end{proof}


\section{Combined Small inertia \& mean field limits: from Newton to Aggregation}

\subsection{Proof of Theorem \ref{thm_main2}}
We first start with the case of smooth interaction potentials as in previous section and apply a similar strategy to the proof of Proposition \ref{lem_mke} to the system \eqref{main_fluid22}. Then we get
\[
\frac{d}{dt} \me^N (\mz^N(t) | \bar U(t)) =: \frac{1}{\e_N}\lt(\sum_{i=1}^4 \bar I_i \rt) + \bar I_5, 
\]
where $\bar I_i, i=1,2,3,4$ are the terms $I_i,i=1,2,3,4$ in \eqref{est_mod} with replacing $(\rho,u)$ by $(\bar \rho, \bar u)$, and $\bar I_5$ is given by
\[
\bar I_5 := \frac1N \sum_{i=1}^N (\bar u(x_i) - v_i) \cdot \bar e.
\]
This can be simply estimated as
$$\begin{aligned}
|\bar I_5| \leq \|\bar e\|_{L^\infty}\frac1N \sum_{i=1}^N |\bar u(x_i) - v_i| 
\leq \frac{C}{\e_N}\frac1N\sum_{i=1}^N |\bar u(x_i) - v_i|^2 + C\e_N
\leq  \frac{C}{\e_N}\me^N (\mz^N(t) | \bar U(t)) + C\e_N.
\end{aligned}$$
where $C>0$ depends only $\|\bar e\|_{L^\infty}$, independent of $N$ and $\e_N$. For the rest, we employ almost the same arguments as before to have
$$\begin{aligned}
\frac{1}{\e_N}\lt( \sum_{i=1}^4 \bar I_i  \rt)
&\leq  - \frac{2\gamma}{\e_N} \me^N (\mz^N(t) | \bar U(t))- \frac1{\e_N N} \sum_{i=1}^N \psi(x_i - y_i)|v_i - \bar u(x_i)|^2\cr
&\quad  + \frac C{\e_N}\me^N (\mz^N(t) | \bar U(t)) + C d_{BL}^2(\rho^N_t(\cdot), \rho(\cdot,t))  ,
\end{aligned}$$
where $C>0$ is independent of $N$, $\e_N$, and $\gamma > 0$. This yields
\bq\label{est_si}
\frac{d}{dt} \me^N (\mz^N(t) | \bar U(t)) + \frac{2\gamma - C}{\e_N} \me^N (\mz^N(t) | \bar U(t))  \leq \frac{C}{\e_N}d_{BL}^2(\rho^N_t(\cdot), \rho(\cdot,t))    + C\e_N,
\eq
where $C>0$ is independent of $N$, $\e_N$, and $\gamma > 0$. On the other hand, by Proposition \ref{lem_dbl}, we can bound the first term on the right hand side of the above inequality from above by
\[
\frac{C}{\e_N}d_{BL}^2(\rho^N_0,\bar \rho_0) + \frac{C}{\e_N}\int_0^t \me^N (\mz^N(s) | \bar U(s))\,ds,
\]
where $C>0$ is independent of $N$, $\e_N$, and $\gamma > 0$. This together with integrating \eqref{est_si} in time implies
$$\begin{aligned}
\me^N (\mz^N(t) | \bar U(t)) + \frac{2\gamma - C}{\e_N} \int_0^t \me^N (\mz^N(s) | \bar U(s))\,ds 
& + \frac1{\e_N N} \sum_{i=1}^N \int_0^t \psi(x_i(s) - y_i(s))|v_i(s) - \bar u(x_i(s),s)|^2\,ds\cr
& \leq \me^N (\mz^N_0 | \bar U_0)  + \frac{C}{\e_N}d_{BL}^2(\rho^N_0,\bar \rho_0)  + C\e_N.
\end{aligned}$$
We finally apply Gr\"onwall's lemma to conclude the desired result in Theorem \ref{thm_main2}. 

\subsection{Singular interaction potential cases}
Similarly as before, Theorem \ref{thm_main2} can be also easily extended to the case with Coulomb or Riesz potentials $\wt W$ defined in \eqref{w_a} or \eqref{w_a2}. More specifically, we have the following theorem.

\begin{theorem}\label{thm_main2s} Let $T>0$ and $\mz^N(t) = \{(x_i(t), v_i(t))\}_{i=1}^N$ be a solution to the particle system \eqref{main_par}, and let $(\bar \rho, \bar u)$ be the unique strong solution of the aggregation-type equation \eqref{eq_agg}--\eqref{eq_agg2} with $\wt W$, which is appeared in \eqref{w_a} or \eqref{w_a2}, instead of $W$, in the sense of Definition \ref{def_strong2} up to time $T>0$ with the initial data $\bar \rho_0$. Suppose that the strength of damping $\gamma >0$ is large enough and $(\bar\rho,\bar u)$ satisfies $\bar\rho \in L^\infty(\R^d \times (0,T))$. We further assume that $\bar\rho \in L^\infty(0,T;\mc^\sigma(\R^d))$ for some $\sigma > \alpha-d+1$ in the case $s \geq d-1$. Then there exists $\beta < 2$ such that 
$$\begin{aligned}
&d_{BL}^2(\rho^N_t(\cdot),\bar \rho(\cdot,t))  +  \int_{\R^d \times \R^d \setminus \Delta} \,\wt W(x-y)(\rho^N - \bar\rho)(x)(\rho^N - \bar\rho)(y)\,dxdy \cr
&\quad 
+\int_0^t \int_{\R^d \times \R^d} |v - \bar u(x,s) |^2\mu^N_s(dxdv) \,ds\cr
&\qquad  \leq  Cd_{BL}^2(\rho^N_0,\bar \rho_0) + C \int_{\R^d \times \R^d \setminus \Delta} \wt W(x-y)(\rho^N_0 - \bar\rho_0)(x)(\rho^N_0 - \bar\rho_0)(y)\,dxdy   \cr
&\qquad \quad + C\e_N\int_{\R^d \times \R^d} |v - \bar u_0(x) |^2\mu^N_0(dxdv)  + C\e_N^2 + CN^{\beta-2} \cr
\end{aligned}$$
and
$$\begin{aligned}
&\frac1{\e_N}d^2_{BL}(\rho^N_t(\cdot),\bar \rho(\cdot,t))  +  \frac1{\e_N}\,\int_{\R^d \times \R^d \setminus \Delta} \wt W(x-y)(\rho^N - \bar\rho)(x)(\rho^N - \bar\rho)(y)\,dxdy  \cr
&\quad + \int_{\R^d \times \R^d} |v - \bar u(x,t) |^2\mu^N_t(dxdv)  \cr
&\qquad \leq  \frac{C}{\e_N}d_{BL}^2(\rho^N_0,\bar \rho_0)    + \frac{C}{\e_N} \int_{\R^d \times \R^d \setminus \Delta} \wt W(x-y)(\rho^N_0 - \bar\rho_0)(x)(\rho^N_0 - \bar\rho_0)(y)\,dxdy \cr
&\qquad \quad + C(1 + \e_N)\int_{\R^d \times \R^d} |v - \bar u_0(x) |^2\mu^N_0(dxdv) + C\e_N + C\frac{N^{\beta-2}}{\e_N}
\end{aligned}$$
for all $t \in [0,T]$, where $C>0$ is independent of $\e_N$ and $N$.
In particular if 
\[
\int_{\R^d \times \R^d} |v - \bar u_0(x) |^2\mu^N_0(dxdv)  \leq C\e_N
\]
and
\[
d^2_{BL}(\rho^N_0,\bar \rho_0) +  \int_{\R^d \times \R^d \setminus \Delta} \wt W(x-y)(\rho^N_0 - \bar\rho_0)(x)(\rho^N_0 - \bar\rho_0)(y)\,dxdy \leq C\e_N^2
\]
for some $C>0$ which is independent of $\e_N$, then we have
$$\begin{aligned}
&d_{BL}^2(\rho^N_t(\cdot),\bar \rho(\cdot,t))  + \int_{\R^d \times \R^d \setminus \Delta} \wt W(x-y)(\rho^N - \bar\rho)(x)(\rho^N - \bar\rho)(y)\,dxdy \leq C\e_N^2  + CN^{\beta-2}
\end{aligned}$$
and
\[
\int_{\R^d \times \R^d} |v - \bar u(x,t) |^2\mu^N_t(dxdv) \leq C\e_N + C\frac{N^{\beta-2}}{\e_N}
\]
for all $t \in [0,T]$, where $C>0$ is independent of $\e_N$ and $N$.
\end{theorem}

%
%
%
%
\section{Local Cauchy problem for pressureless Euler equations with nonlocal forces}

In order to make the analysis for the mean-field limit from the particle system \eqref{main_par} to the pressureless Euler-type equations \eqref{main_fluid} fully rigorous, we need to have the existence of solutions for both systems. As mentioned in Introduction, we postpone the existence theory for the particle system \eqref{main_par} in Appendix \ref{app_par}, and here we provide local-in-time existence and uniqueness of classical solutions for the system \eqref{main_fluid}. For the reader's convenience, let us recall our limiting system:
\begin{align}\label{main_fluid2}
\begin{aligned}
&\pa_t \rho + \nabla_x \cdot (\rho u) = 0, \quad (x,t) \in \R^d \times \R_+,\cr
&\pa_t (\rho u) + \nabla_x \cdot (\rho u \otimes u) = -\rho u - \rho \nabla_x V - \rho \nabla_x W \star\rho \cr
&\hspace{4cm} + \rho \int_{\R^d} \psi(x-y) (u(y) - u(x))\,\rho(y)\,dy,
\end{aligned}
\end{align}
with the initial data:
\[
(\rho(x,t),u(x,t))|_{t=0} =: (\rho_0(x), u_0(x)), \quad x \in \R^d.
\]
Here we set the coefficient of linear damping $\gamma=1$. 

We first introduce the exact notion of strong solution to the system \eqref{main_fluid2} that we will deal with.
\begin{definition}\label{def_strong3} Let $s > d/2+1$. For given $T\in(0,\infty)$, the pair $(\rho,u)$ is a strong solution of \eqref{main_fluid2} on the time interval $[0,T]$ if and only if the following conditions are satisfied:
\begin{itemize}
\item[(i)] $\rho \in \mc([0,T];H^s(\R^d))$, $u \in \mc([0,T];Lip(\R^d)\cap L^2_{loc}(\R^d))$, and $\nabla_x^2 u \in \mc([0,T];H^{s-1}(\R^d))$,
\item[(ii)] $(\rho, u)$ satisfy the system \eqref{main_fluid2} in the sense of distributions.
\end{itemize}
\end{definition}
Notice that due to the choice of $s$ in the previous definition, these strong solutions are also classical solutions to \eqref{main_fluid2} in the sense of Definition \ref{def_strong}. Our main result of this section is the following local Cauchy problem for the system \eqref{main_fluid2}.

\begin{theorem}\label{thm_local}Let $s > d/2+1$ and $R>0$. Suppose that the confinement potential $V$ is given by $V = |x|^2/2$, the interaction potential $\nabla_x W \in (\W^{1,1} \cap \W^{1,\infty})(\R^d)$, and the communication weight function $\psi$ satisfies 
\bq\label{as_psi}
\psi \in \mc_c^1(\R^d) \quad \mbox{and} \quad supp(\psi) \subseteq B(0,R),
\eq
where $B(0,R) \subset \R^d$ denotes a ball of radius $R$ centered at the origin. For any $N<M$, there is a positive constant $T^*$ depending only on $R$, $N$, and $M$ such that if $\rho_0 > 0$ on $\R^d$ and 
\[
\|\rho_0\|_{H^s} + \|u_0\|_{L^2(B(0,R))}+\|\nabla_x u_0\|_{L^\infty} +  \|\nabla_x^2 u_0\|_{H^{s-1}} < N,
\]
then the Cauchy problem \eqref{main_fluid2} has a unique strong solution $(\rho,u)$, in the sense of Definition \ref{def_strong3}, satisfying
\[
\sup_{0 \leq t \leq T^*}\lt(\|\rho(\cdot,t)\|_{H^s} + \|u(\cdot,t)\|_{L^2(B(0,R))} +\|\nabla_x u(\cdot,t)\|_{L^\infty} + \|\nabla_x^2 u(\cdot,t)\|_{H^{s-1}}\rt) \leq M.
\]
\end{theorem}

\begin{remark}The assumption on the communication weight function \eqref{as_psi} implies
$
\psi \in \W^{1,p}(\R^d)
$
for any $p \in [1,\infty]$.
\end{remark}

\begin{remark} The $L^2$-norm of $u$ on the ball is introduced due to the confinement potential $V$. In fact, if we ignore the confinement potential $V$ in the momentum equation in \eqref{main_fluid2}, then under the following assumption on the initial data
\[
\|\rho_0\|_{H^s} + \|u_0\|_{H^{s+1}} < N,
\]
we  have the unique strong solution $(\rho,u)$ to the system \eqref{main_fluid2} satisfying 
\[
\sup_{0 \leq t \leq T^*}\lt(\|\rho(\cdot,t)\|_{H^s} + \|u(\cdot,t)\|_{H^{s+1}}\rt) \leq M.
\]
\end{remark}

\subsection{Linearized system}
In this part, we construct approximate solutions $(\rho^n, u^n)$ for the system \eqref{main_fluid2} and provide some uniform bound estimates of it.

Let us first take the initial data as the zeroth approximation:
 \[
(\rho^0(x,t),u^0(x,t)) = (\rho_0(x),u_0(x)), \quad (x,t) \in \R^d \times \R_+.
\]
We next suppose that the $n$th approximation $(\rho^n, u^n)$ with $n \geq 1$ is given. Then we define the $(n+1)$th approximation $(\rho^{n+1}, u^{n+1})$  as a solution to the following linear system. 
\begin{align}\label{main_lin}
\begin{aligned}
&\pa_t \rho^{n+1} + u^n \cdot \nabla \rho^{n+1} + \rho^{n+1} \nabla \cdot u^n = 0, \quad (x,t) \in \R^d \times \R_+,\cr
&\rho^{n+1}\pa_t u^{n+1} + \rho^{n+1} u^n \cdot \nabla u^{n+1} = - \rho^{n+1}u^{n+1} - \rho^{n+1}(\nabla_x V + \nabla_x W \star \rho^{n+1})\cr
&\hspace{5cm} + \rho^{n +1} \intr \psi(x-y) (u^n(y) - u^n(x)) \rho^{n+1}(y)\,dy,
\end{aligned}
\end{align}
with the initial data 
\[
(\rho^n(x,0),u^n(x,0))=(\rho_0(x),u_0(x)) \quad \mbox{for all} \quad n \geq 1, \quad x\in \R^d.
\]
Let us introduce a solution space $\mathcal{Y}_{s,R}(T)$ with $s > d/2+1$ as
$$\begin{aligned}
\mathcal{Y}_{s,R}(T) &:= \Big\{ (\rho,u) : \rho \in \mc([0,T];H^s(\R^d)), u \in \mc([0,T];L^2(B(0,R))) \cap \mc([0,T];\dot{\W}^{1,\infty}(\R^d)), \cr
&\hspace{4cm} \nabla_x^2 u \in \mc([0,T];H^{s-1}(\R^d)) \Big\}.
\end{aligned}$$
Then by the standard linear solvability theory \cite{K73}, for any $T>0$ we have that the approximation $\{(\rho^n,u^n)\}_{n=0}^\infty \subset \mathcal{Y}_{s,R}(T)$ is well-defined. 

For notational simplicity, in the rest of this section, we drop $x$-dependence of the differential operator $\nabla_x$.

\begin{proposition}Suppose that the initial data $(\rho_0, u_0)$ satisfies $\rho_0 > 0$ on $\R^d$ and 
\[
\|\rho_0\|_{H^s} + \|u_0\|_{L^2(B(0,R))} + \|\nabla u_0\|_{L^\infty} +  \|\nabla^2 u_0\|_{H^{s-1}} < N,
\]
and let $(\rho^n,u^n)$ be a sequence of the approximate solutions of \eqref{main_lin} with the initial data $(\rho_0, u_0)$. Then for any $N < M$, there exists $T^* > 0$ such that 
\[
\sup_{n \geq 0} \sup_{0 \leq t \leq T^*}\lt(\|\rho^n(\cdot,t)\|_{H^s} + \|u^n(\cdot,t)\|_{L^2(B(0,R))} + \|\nabla u^n(\cdot,t)\|_{L^\infty} + \|\nabla^2 u^n(\cdot,t)\|_{H^{s-1}}\rt) \leq M.
\]
\end{proposition}
\begin{proof}For the proof, we use the inductive argument. Since we take the initial data for the first iteration step, it is clear to find
$$\begin{aligned}
&\sup_{0 \leq t \leq T}\lt(\|\rho^0(\cdot,t)\|_{H^s} + \|u^0(\cdot,t)\|_{L^2(B(0,R))} + \|\nabla u^0(\cdot,t)\|_{L^\infty} + \|\nabla^2 u^0(\cdot,t)\|_{H^{s-1}}\rt)\cr
&\quad = \|\rho_0\|_{H^s} + \|u_0\|_{L^2(B(0,R))} + \|\nabla u_0\|_{L^\infty} +  \|\nabla^2 u_0\|_{H^{s-1}}  < N< M.
\end{aligned}$$
We now suppose that 
\[
\sup_{0 \leq t \leq T_0}\lt(\|\rho^n(\cdot,t)\|_{H^s} + \|u^n(\cdot,t)\|_{L^2(B(0,R))} + \|\nabla u^n(\cdot,t)\|_{L^\infty} + \|\nabla^2 u^n(\cdot,t)\|_{H^{s-1}}\rt) \leq M
\]
for some $T_0 > 0$. In the rest of the proof, upon mollifying if necessary we may assume that the communication weight function $\psi$ is smooth. Since this proof is a rather lengthy, we divide it into four steps:

\begin{itemize}
\item In {\bf Step A}, we provide the positivity and $H^s(\R^d)$-estimate of $\rho^{n+1}$:
\[
\rho^{n+1}(x,t) > 0 \quad  \forall\, (x,t) \in \R^d \times [0,T] \quad \mbox{and} \quad \|\rho^{n+1}(\cdot,t)\|_{H^s} \leq \|\rho_0\|_{H^s}e^{CMt}
\]
for $t \leq T_0$, where $C>0$ is independent of $n$. \newline

\item In {\bf Step B}, we show $\dot \W^{1,\infty}(\R^d)$-estimate and $L^2(B(0,R))$-estimate of $u^{n+1}$:
$$\begin{aligned}
&\|\nabla u^{n+1}(\cdot,t)\|_{L^\infty} + \|u^{n+1}(\cdot,t)\|_{L^2(B(0,R))} \leq \|\nabla u_0\|_{L^\infty}e^{(CM - 1)t} + \|u_0\|_{L^2(B(0,R))} + E(t)
\end{aligned}$$
for $t \leq T_0$, where $C>0$ is independent of $n$, and $E: [0,T_0] \to [0,\infty)$ is continuous on $[0,T_0]$ satisfying $E(t) \to 0$ as $t \to 0^+$.  \newline

\item In {\bf Step C}, we estimate the higher order derivative of $u^{n+1}$:
\[
\|\nabla^2 u^{n+1}\|_{H^{s-1}} \leq \|\nabla^2 u_0\|_{H^{s-1}} e^{CMt} + E(t)
\]
for $t \leq T_0$, where $C>0$ is independent of $n$, and $E$ satisfies the same property as in {\bf Step B}. \newline

\item In {\bf Step D}, we finally combine all of the estimates in {\bf Steps A, B, \& C} to conclude our desired result. \newline
\end{itemize}

{\bf Step A.-} We first show the positivity of $\rho^{n+1}$. Consider the following characteristic flow $\eta^{n+1}$ associated to the fluid velocity $u^n$ by
\bq\label{char_eq}
\pa_t \eta^{n+1}(x,t) = u^n(\eta^{n+1}(x,t),t) \quad \mbox{for} \quad t > 0
\eq
with the initial data 
$
\eta^{n+1}(x,0) =x \in \R^d.
$
Since $u^n$ is globally Lipschitz, the characteristic equations \eqref{char_eq} are well-defined. Then by using the method of characteristics, we obtain 
\[
\pa_t \rho^{n+1}(\eta^{n+1}(x,t),t) =  - \rho^{n+1}(\eta^{n+1}(x,t),t) (\nabla \cdot u)(\eta^{n+1}(x,t),t),
\]
and applying Gr\"onwall's lemma yields
\[
\rho^{n+1}(\eta^{n+1}(x,t),t) = \rho_0(x) \exp\lt(-\int_0^t (\nabla \cdot u)(\eta^{n+1}(x,\tau),\tau)\,d\tau \rt) \geq \rho_0(x) e^{-MT_0} > 0.
\]
We next estimate $H^s$-norm of $\rho^{n+1}$. We first easily find 
$$\begin{aligned}
\frac{d}{dt}\|\rho^{n+1}\|_{L^2}^2 &\leq C\|\nabla u^n\|_{L^\infty}\|\rho^{n+1}\|_{L^2}^2\leq CM\|\rho^{n+1}\|_{L^2}^2,\cr
\frac{d}{dt}\|\nabla \rho^{n+1}\|_{L^2}^2 &\leq C\|\nabla u^n\|_{L^\infty}\|\nabla \rho^{n+1}\|_{L^2}^2 + C\|\nabla^2 u^n\|_{L^2}\|\rho^{n+1}\|_{L^\infty}\|\nabla \rho^{n+1}\|_{L^2} \leq CM\|\rho^{n+1}\|_{H^s}\|\nabla \rho^{n+1}\|_{L^2},
\end{aligned}$$
and 
$$\begin{aligned}
&\frac12\frac{d}{dt}\int_{\R^d} |\nabla^k \rho^{n+1}|^2\,dx \cr
&\quad = - \int_{\R^d} \nabla^k \rho^{n+1} \cdot (u^n \cdot \nabla^{k+1} \rho^{n+1})\,dx - \int_{\R^d} \nabla^k \rho^{n+1} \cdot (\nabla^k (\nabla \rho^{n+1} \cdot u^n) - u^n \cdot \nabla^{k+1} \rho^{n+1})\,dx\cr
&\qquad - \int_{\R^d} \nabla \rho^{n+1} \cdot (\nabla^k (\nabla \cdot u^n)) \rho^{n+1}\,dx - \int_{\R^d} \nabla^k \rho^{n+1} \cdot (\nabla^k(\rho^{n+1} \nabla \cdot u^n) -  \rho\nabla^k (\nabla \cdot u^n))\,dx\cr
&\quad =: \sum_{i=1}^4 I_i
\end{aligned}$$
for $2 \leq k \leq s$. Here we use Moser-type inequality to estimate $I_i,i=1,\cdots,4$ as
$$\begin{aligned}
I_1 &\leq \|\nabla u^n\|_{L^\infty}\|\nabla^k \rho^{n+1}\|_{L^2}^2\leq CM\|\nabla^k \rho^{n+1}\|_{L^2}^2,\cr
I_2 &\leq \|\nabla^k (\nabla \rho^{n+1} \cdot u^n) - u^n \cdot \nabla^{k+1} \rho^{n+1}\|_{L^2}\|\nabla^k\rho^{n+1}\|_{L^2}\cr
&\leq C\lt(\|\nabla^k u^n\|_{L^2}\|\nabla \rho^{n+1}\|_{L^\infty} +  \|\nabla u^n\|_{L^\infty}\|\nabla^k\rho^{n+1}\|_{L^2}\rt)\|\nabla^k\rho^{n+1}\|_{L^2}\cr
&\leq CM\|\nabla \rho^{n+1}\|_{H^{s-1}}\|\nabla^k\rho^{n+1}\|_{L^2},\cr
I_3 &\leq \|\rho^{n+1}\|_{L^\infty}\|\nabla^k\rho^{n+1}\|_{L^2}\|\nabla^{k+1} u^n\|_{L^2}\leq CM\|\rho^{n+1}\|_{H^s}\|\nabla^k\rho^{n+1}\|_{L^2},\cr
I_4 &\leq \|\nabla^k(\rho^{n+1} \nabla \cdot u^n) -  \rho^{n+1}\nabla^k (\nabla \cdot u^n)\|_{L^2}\|\nabla^k\rho^{n+1}\|_{L^2}\cr
&\leq C\lt(\|\nabla^k \rho^{n+1}\|_{L^2}\|\nabla u^n\|_{L^\infty} + \|\nabla \rho^{n+1}\|_{L^\infty}\|\nabla^k u^n\|_{L^2} \rt)\|\nabla^k\rho^{n+1}\|_{L^2}\cr
&\leq CM\|\nabla \rho^{n+1}\|_{H^{s-1}}\|\nabla^k\rho^{n+1}\|_{L^2}.
\end{aligned}$$
Combining all of the above estimates entails
\bq\label{est_lrho}
\frac{d}{dt}\|\rho^{n+1}\|_{H^s} \leq CM\|\rho^{n+1}\|_{H^s}, \quad \mbox{i.e.,} \quad  \|\rho^{n+1}(\cdot,t)\|_{H^s} \leq \|\rho_0\|_{H^s}e^{CMt},
\eq
for $t \leq T_0$, where $C > 0$ is independent of $n$. \newline

{\bf Step B.-} Due to the positivity of $\rho^{n+1}$, it follows from the momentum equation in \eqref{main_lin} that $u^{n+1}$ satisfies
\begin{align}\label{mom_lin}
\begin{aligned}
\pa_t u^{n+1} +  u^n \cdot \nabla u^{n+1} &= - u^{n+1} -  \nabla V - \nabla W \star \rho^{n+1}+  \intr \psi(x-y) (u^n(y) - u^n(x)) \rho^{n+1}(y)\,dy.
\end{aligned}
\end{align}
Taking the differential operator $\nabla$ to \eqref{mom_lin} gives
\begin{align}\label{na_u}
\begin{aligned}
\pa_t \nabla u^{n+1} +  u^n \cdot \nabla^2 u^{n+1}  &= - \nabla u^n \,\nabla u^{n+1} - \nabla u^{n+1} -  \mathbb{I}_d - \nabla W \star \nabla \rho^{n+1}\cr
&\quad +  \intr (u^n(y) - u^n(x)) \otimes \nabla_x \psi(x-y) \rho^{n+1}(y)\,dy\cr
&\quad - \nabla u^n \intr \psi(x-y) \rho^{n+1}(y)\,dy,
\end{aligned}
\end{align}
where we used $\nabla V = x$ and $\mathbb{I}_d$ denotes the $n \times n$ identity matrix. Note that
\[
|\nabla u^n \,\nabla u^{n+1}| \leq M\|\nabla u^{n+1}(\cdot,t)\|_{L^\infty}
\]
and
\[
\|\nabla W \star \nabla \rho^{n+1}\|_{L^\infty} \leq \|\nabla W\|_{L^2} \|\nabla \rho^{n+1}\|_{L^2}.
\]
We also estimate the last terms on the right hand side of \eqref{na_u} as
$$\begin{aligned}
\lt|\intr (u^n(y) - u^n(x)) \otimes \nabla_x \psi(x-y) \rho^{n+1}(y)\,dy \rt| 
&\leq \int_{|x-y| \leq R} |u^n(y) - u^n(x)| |\nabla_x \psi(x-y)| \rho^{n+1}(y)\,dy \cr
&\leq \|\nabla u^n\|_{L^\infty} \int_{|x-y| \leq R} |y-x| |\nabla_x \psi(x-y)| \rho^{n+1}(y)\,dy\cr
&\leq \|\nabla u^n\|_{L^\infty} R \|\nabla \psi\|_{L^2}\|\rho^{n+1}\|_{L^2}\cr
&\leq CM\|\psi\|_{L^2}\|\rho^{n+1}\|_{L^2}
\end{aligned}$$
and
\[
\lt|\nabla u^n \intr \psi(x-y) \rho^{n+1}(y)\,dy \rt| \leq  \|\nabla u^n\|_{L^\infty} \|\psi\|_{L^2}\|\rho^{n+1}\|_{L^2} \leq CM\|\psi\|_{L^2}\|\rho^{n+1}\|_{L^2}.
\]
These estimates together with integrating \eqref{na_u} along the characteristic flow $\eta^{n+1}$ implies
$$\begin{aligned}
e^t\|\nabla u^{n+1}(\cdot,t)\|_{L^\infty} &\leq \|\nabla u_0\|_{L^\infty} + CM\int_0^t e^\tau \|\nabla u^{n+1}(\cdot,\tau)\|_{L^\infty}\,d\tau  + C(1+M)\int_0^t e^\tau \|\rho^{n+1}(\cdot,\tau)\|_{H^s}\,d\tau.
\end{aligned}$$
By using Gr\"onwall's lemma, we obtain
$$\begin{aligned}
e^t\|\nabla u^{n+1}(\cdot,t)\|_{L^\infty} &\leq \|\nabla u_0\|_{L^\infty}e^{CMt} + C(1+M)\int_0^t e^\tau \|\rho^{n+1}(\cdot,\tau)\|_{H^s}\,d\tau  \cr
&\quad + CM(1+M) e^{CMt}\int_0^t e^{-CM\xi}\int_0^\xi e^\tau \|\rho^{n+1}(\cdot,\tau)\|_{H^s}\,d\tau d\xi.
\end{aligned}$$
This together with \eqref{est_lrho} asserts
\bq\label{u_inf}
\|\nabla u^{n+1}(\cdot,t)\|_{L^\infty} \leq \|\nabla u_0\|_{L^\infty}e^{(CM - 1)t} + E_1(t),
\eq
where $E_1: [0,T_0] \to [0,\infty)$ is continuous on $[0,T_0]$ satisfying $E_1(t) \to 0$ as $t \to 0^+$. 

For the $L^2$-estimate of $u^{n+1}$ on $B(0,R)$, we multiply \eqref{mom_lin} by $u^{n+1}$ and integrate it over $B(0,R)$ to yield
$$\begin{aligned}
\frac12\frac{d}{dt}\intb |u^{n+1}|^2\,dx
&= \intb u^{n+1} \cdot \lt( - u^n \cdot \nabla u^{n+1} - u^{n+1} - \nabla V - \nabla W \star \rho^{n+1}\rt) dx \cr
&\quad + \intb u^{n+1} \cdot \lt(\intr \psi(x-y) (u^n(y) - u^n(x)) \rho^{n+1}(y)\,dy\rt)dx\cr
&\leq \|\nabla u^{n+1}\|_{L^\infty} \|u^n\|_{L^2(B(0,R))}\|u^{n+1}\|_{L^2(B(0,R))} - \|u^{n+1}\|_{L^2(B(0,R))}^2\cr
&\quad + R\|u^{n+1}\|_{L^1(B(0,R))} + C(\|\rho^{n+1}\|_{L^2} + \|\rho^{n+1}\|_{L^\infty}) \|u^{n+1}\|_{L^1(B(0,R))}\cr
&\quad +\|\nabla u^n\|_{L^\infty} R \|\psi\|_{L^2}\|\rho^{n+1}\|_{L^2} \|u^{n+1}\|_{L^1(B(0,R))}.
\end{aligned}$$
Here we used 
$$\begin{aligned}
\lt|\intr \psi(x-y) (u^n(y) - u^n(x)) \rho^{n+1}(y)\,dy\rt| &\leq \int_{|x-y| \leq R} \psi(x-y) |u^n(y) - u^n(x)| \rho^{n+1}(y)\,dy\cr
&\leq \|\nabla u^n\|_{L^\infty}\int_{|x-y| \leq R} \psi(x-y) |x-y| \rho^{n+1}(y)\,dy\cr
&\leq \|\nabla u^n\|_{L^\infty} R \|\psi\|_{L^2}\|\rho^{n+1}\|_{L^2}.
\end{aligned}$$
Thus we obtain
\[
\frac{d}{dt}\|u^{n+1}\|_{L^2(B(0,R))} \leq CM\|\nabla u^{n+1}\|_{L^\infty} + C(1 + (1+M)\|\rho^{n+1}\|_{H^s} ),
\]
where $C>0$ depends only on $R$ and $\|\psi\|_{L^2}$. Integrating this over $[0,t]$ with $t \leq T_0$ and using the estimates \eqref{est_lrho} and \eqref{u_inf} imply
\bq\label{u_l2}
\|u^{n+1}\|_{L^2(B(0,R))} \leq \|u_0\|_{L^2(B(0,R))} + E_2(t),
\eq
where $E_2: [0,T_0] \to [0,\infty)$ is continuous on $[0,T_0]$ satisfying $E_2(t) \to 0$ as $t \to 0^+$. \newline

{\bf Step C.-} For $2 \leq k \leq s+1$, we find
$$\begin{aligned}
&\frac12\frac{d}{dt}\int_{\R^d} |\nabla^k u^{n+1}|^2\,dx \cr
&\quad = - \int_{\R^d} \nabla^k u^{n+1} \cdot (u^n \cdot \nabla^{k+1} u^{n+1})\,dx - \int_{\R^d} \nabla^k u \cdot ( \nabla^k(u^n \cdot \nabla u^{n+1}) - u^n \cdot \nabla^{k+1} u^{n+1})\,dx\cr
&\qquad - \int_{\R^d} |\nabla^k u^{n+1}|^2\,dx - \int_{\R^d} \nabla^k u^{n+1} \cdot \nabla^k (\nabla W \star \rho^{n+1})\,dx\cr
&\qquad + \intr \nabla^k u^{n+1} \cdot \nabla^k \intr \psi(x-y) (u^n(y) - u^n(x)) \rho^{n+1}(y)\,dydx\cr
&\quad =: \sum_{k=1}^5 J_k,
\end{aligned}$$
where $J_1$ and $J_2$ can be estimated as
\[
J_1 \leq \|\nabla u^n\|_{L^\infty}\|\nabla^k u^{n+1}\|_{L^2}^2 \leq M\|\nabla^k u^{n+1}\|_{L^2}^2
\]
and
$$\begin{aligned}
J_2 &\leq C\lt(\|\nabla^k u^n\|_{L^2}\|\nabla u^{n+1}\|_{L^\infty} + \|\nabla u^n\|_{L^\infty} \|\nabla^k u^{n+1}\|_{L^2} \rt)\|\nabla^k u^{n+1}\|_{L^2} \cr
&\leq CM(\|\nabla u^{n+1}\|_{L^\infty}+ \|\nabla^k u^{n+1}\|_{L^2})\|\nabla^k u^{n+1}\|_{L^2}.
\end{aligned}$$
For the estimate of $J_4$, we use the fact that $W$ is the Coulombian potential to deduce 
$$\begin{aligned}
\lt|J_4\rt| &= \lt| \int_{\R^d} |\nabla^{k} u^{n+1}||\nabla^2 W \star \nabla^{k-1}\rho^{n+1}|\,dx\rt| \leq \|\nabla^k u^{n+1}\|_{L^2}\|\nabla^2 W\|_{L^1} \|\nabla^{k-1} \rho^{n+1}\|_{L^2}.
\end{aligned}$$
We next divide $J_5$ into two terms:
$$\begin{aligned}
J_5 &= \sum_{0 \leq \ell \leq k} \binom{k}{\ell}\intrr \nabla^k u^{n+1}(x) \nabla_x^\ell \psi(x-y) \nabla_x^{k-\ell}(u^n(y) - u^n(x)) \rho^{n+1}(y)\,dydx\cr
&= -\sum_{0 \leq \ell \leq k-1} \binom{k}{\ell}\intrr \nabla^k u^{n+1}(x) \nabla_x^\ell \psi(x-y) \nabla_x^{k-\ell}u^n(x) \rho^{n+1}(y)\,dydx\cr
&\quad + \intrr \nabla^k u^{n+1}(x) \nabla_x^k \psi(x-y) (u^n(y) - u^n(x)) \rho^{n+1}(y)\,dydx\cr
&=: J_5^1 + J_5^2.
\end{aligned}$$
Note that 
$$\begin{aligned}
&\lt|\intrr \nabla^k u^{n+1}(x) \nabla_x^\ell \psi(x-y) \nabla_x^{k-\ell}u^n(x) \rho^{n+1}(y)\,dydx\rt|\cr
&\quad = \lt|\intrr \nabla^k u^{n+1}(x) \nabla_y^\ell \psi(x-y) \nabla_x^{k-\ell}u^n(x) \rho^{n+1}(y)\,dydx\rt|\cr
&\quad = \lt|\intrr \psi(x-y)\nabla^k u^{n+1}(x)  \nabla_x^{k-\ell}u^n(x) \nabla_y^\ell \rho^{n+1}(y)\,dydx\rt|.
\end{aligned}$$
Thus for $\ell = k-1$ we get
$$\begin{aligned}
&\lt|\intrr  \psi(x-y) \nabla^k u^{n+1}(x) \nabla u^n(x)  \nabla_y^{k-1}\rho^{n+1}(y)\,dydx\rt|\cr
&\quad \leq \|\nabla u^n\|_{L^\infty} \intrr \psi(x-y)| \nabla^k u^{n+1}(x)|  |\nabla^{k-1}\rho^{n+1}(y)|\,dydx\cr
&\quad \leq \|\nabla u^n\|_{L^\infty}\|\psi\|_{L^1}\|\nabla^k u^{n+1}\|_{L^2}\|\nabla^{k-1}\rho^{n+1}\|_{L^2}\cr
&\quad \leq CM\|\nabla^k u^{n+1}\|_{L^2}\|\nabla^{k-1}\rho^{n+1}\|_{L^2}
\end{aligned}$$
and for $0 \leq \ell \leq k-2$ we obtain
$$\begin{aligned}
&\lt|\intrr \psi(x-y)\nabla^k u^{n+1}(x)  \nabla_x^{k-\ell}u^n(x) \nabla_y^\ell \rho^{n+1}(y)\,dydx\rt|\cr
&\quad \leq \|\nabla^k u^{n+1}\|_{L^2}\| \nabla^{k-\ell}u^n\|_{L^2}\|\psi\|_{L^2} \|\nabla^\ell \rho^{n+1}\|_{L^2}\cr
&\quad \leq CM\|\nabla^k u^{n+1}\|_{L^2}\|\nabla^\ell \rho^{n+1}\|_{L^2}.
\end{aligned}$$
This asserts
$$\begin{aligned}
J_5^1 &\leq CM\|\nabla^k u^{n+1}\|_{L^2}\sum_{0 \leq \ell \leq k-2} \binom{k}{\ell} \|\nabla^\ell \rho^{n+1}\|_{L^2} + CM\|\nabla^k u^{n+1}\|_{L^2}\|\nabla^{k-1}\rho^{n+1}\|_{L^2}\cr
&\leq CM\|\nabla^k u^{n+1}\|_{L^2} \|\rho^{n+1}\|_{H^{k-1}}.
\end{aligned}$$
Similarly, by integration by parts, we notice that 
$$\begin{aligned}
&\lt|\intrr \nabla^k u^{n+1}(x) \nabla_x^k \psi(x-y) (u^n(y) - u^n(x)) \rho^{n+1}(y)\,dydx\rt|\cr
&\quad = \lt|\intrr \nabla^k u^{n+1}(x) \nabla_y^{k-1} \nabla_x \psi(x-y) (u^n(y) - u^n(x)) \rho^{n+1}(y)\,dydx\rt|\cr
&\quad = \lt|\intrr  \nabla^k u^{n+1}(x) \nabla_x \psi(x-y)   \nabla_y^{k-1}\lt((u^n(y) - u^n(x)) \rho^{n+1}(y)\rt)dydx\rt|\cr
&\quad = \lt|\sum_{0 \leq \ell \leq k-1}\binom{k-1}{\ell} \intrr  \nabla^k u^{n+1}(x) \nabla_x \psi(x-y)   \nabla_y^{k-1-\ell}(u^n(y) - u^n(x)) \nabla_y^\ell\rho^{n+1}(y) \,dydx\rt|.
\end{aligned}$$
On the other hand, we find
$$\begin{aligned}
&\lt|\intrr  \nabla^k u^{n+1}(x) \nabla_x \psi(x-y) (u^n(y) - u^n(x)) \nabla_y^{k-1}\rho^{n+1}(y) \,dydx\rt|\cr
&\quad \leq \|\nabla u^n\|_{L^\infty}\int_{|x-y|\leq R} |\nabla^k u^{n+1}(x)|| \nabla_x \psi(x-y)| |x-y| | \nabla_y^{k-1}\rho^{n+1}(y)| \,dydx\cr
&\quad \leq R \|\nabla u^n\|_{L^\infty} \|\psi\|_{L^1}\|\nabla^k u^{n+1}\|_{L^2} \|\nabla^{k-1} \rho^{n+1}\|_{L^2}\cr
&\quad \leq CM\|\nabla^k u^{n+1}\|_{L^2} \|\nabla^{k-1} \rho^{n+1}\|_{L^2},
\end{aligned}$$
and
$$\begin{aligned}
&\lt|\intrr  \nabla^k u^{n+1}(x) \nabla_x \psi(x-y)   \nabla_y u^n(y) \nabla_y^{k-2}\rho^{n+1}(y) \,dydx\rt|\cr
&\quad \leq \|\nabla u^n\|_{L^\infty} \intrr |\nabla^k u^{n+1}(x)|| \nabla_x \psi(x-y)|   |\nabla_y^{k-2}\rho^{n+1}(y)| \,dydx\cr
&\quad \leq \|\nabla u^n\|_{L^\infty}\|\nabla \psi\|_{L^1} \|\nabla^k u^{n+1}\|_{L^2} \|\nabla^{k-2} \rho^{n+1}\|_{L^2}\cr
&\quad \leq CM\|\nabla^k u^{n+1}\|_{L^2} \|\nabla^{k-2} \rho^{n+1}\|_{L^2}.
\end{aligned}$$
Moreover, for $0 \leq \ell \leq k-3$ we obtain
$$\begin{aligned}
&\lt| \intrr  \nabla^k u^{n+1}(x) \nabla_x \psi(x-y)   \nabla_y^{k-1-\ell}u^n(y) \nabla_y^\ell\rho^{n+1}(y) \,dydx\rt|\cr
&\quad \leq \| \nabla^k u^{n+1}\|_{L^2} \|\nabla \psi\|_{L^2} \|\nabla^{k-1-\ell}u^n\|_{L^2} \|\nabla^\ell\rho^{n+1}\|_{L^2}\cr
&\quad \leq CM\| \nabla^k u^{n+1}\|_{L^2}  \|\nabla^\ell\rho^{n+1}\|_{L^2}.
\end{aligned}$$
Thus we have 
$$\begin{aligned}
J_5^2 &\leq CM\| \nabla^k u^{n+1}\|_{L^2}\sum_{0 \leq \ell \leq k-3}\binom{k-1}{\ell}   \|\nabla^\ell\rho^{n+1}\|_{L^2} + CM\|\nabla^k u^{n+1}\|_{L^2} \|\nabla^{k-2} \rho^{n+1}\|_{H^1}\cr
&\leq CM\| \nabla^k u^{n+1}\|_{L^2}\|\rho^{n+1}\|_{H^{k-1}},
\end{aligned}$$
and subsequently we get
\[
J_5 \leq CM\|\nabla^k u^{n+1}\|_{L^2} \|\rho^{n+1}\|_{H^{k-1}}.
\]
We finally combine all of the above estimate to have
\[
\frac{d}{dt}\|\nabla^2 u^{n+1}\|_{H^{s-1}} + \|\nabla^2 u^{n+1}\|_{H^{s-1}} \leq  CM \|\nabla^2 u^{n+1}\|_{H^{s-1}} + CM\|\nabla u^{n+1}\|_{L^\infty} + CM\|\rho^{n+1}\|_{H^s},
\]
and applying Gr\"onwall's lemma gives
\bq\label{high_u}
\|\nabla^2 u^{n+1}\|_{H^{s-1}} \leq \|\nabla^2 u_0\|_{H^{s-1}} e^{CMt} + E_3(t),
\eq
where we used the estimates in {\bf Steps B \& C} and $E_3: [0,T_0] \to [0,\infty)$ is continuous on $[0,T_0]$ satisfying $E_3(t) \to 0$ as $t \to 0^+$. \newline

{\bf Step D.-} We now combine \eqref{est_lrho}, \eqref{u_inf}, \eqref{u_l2}, and \eqref{high_u} to have
\begin{align}\label{final}
\begin{aligned}
&\|\rho^{n+1}(\cdot,t)\|_{H^s} + \|\nabla u^{n+1}(\cdot,t)\|_{L^\infty} + \|u^{n+1}(\cdot,t)\|_{L^2(B(0,R))} + \|\nabla^2 u^{n+1}\|_{H^{s-1}}\cr
&\quad \leq \|\rho_0\|_{H^s}e^{CMt} + \|\nabla u_0\|_{L^\infty}e^{(CM - 1)t} + \|u_0\|_{L^2(B(0,R))} + \|\nabla^2 u_0\|_{H^{s-1}} e^{CMt} + E(t)
\end{aligned}
\end{align}
for $t \leq T_0$, where $C>0$ is independent of $n$, and $E: [0,T_0] \to [0,\infty)$ is continuous on $[0,T_0]$ satisfying $E(t) \to 0$ as $t \to 0^+$. On the other hand, the right hand side of \eqref{final} converges to $\|\rho_0\|_{H^s} + \|u_0\|_{L^2(B(0,R))} + \|\nabla u_0\|_{L^\infty} +  \|\nabla^2 u_0\|_{H^{s-1}}$ as $t \to 0^+$ and that is strictly less than $N$. This asserts that there exists $T_* \leq T_0$ such that 
\[
\sup_{0 \leq t \leq T_*}\|\rho^{n+1}(\cdot,t)\|_{H^s} + \|\nabla u^{n+1}(\cdot,t)\|_{L^\infty} + \|u^{n+1}(\cdot,t)\|_{L^2(B(0,R))} + \|\nabla^2 u^{n+1}\|_{H^{s-1}} \leq M.
\]
This completes the proof.
\end{proof}
%
%
%

\subsection{Proof of Theorem \ref{thm_local}} We first show the existence of a  solution $(\rho,u) \in \mathcal{Y}_{s,R}(T_*)$. Note that $\rho^{n+1} - \rho^n$ and $u^{n+1} - u^n$ satisfy
\begin{align}\label{eqn_rho}
\begin{aligned}
&\pa_t (\rho^{n+1} - \rho^n) + (u^n - u^{n-1})\cdot \nabla \rho^{n+1} + u^{n-1} \cdot \nabla (\rho^{n+1} - \rho^n) \cr
&\qquad + (\rho^{n+1} - \rho^n) \nabla \cdot u^n + \rho^n \nabla \cdot (u^n - u^{n-1}) = 0
\end{aligned}
\end{align}
and
$$\begin{aligned}
&\pa_t (u^{n+1} - u^n) + (u^n - u^{n-1})\cdot \nabla u^{n+1} + u^{n-1} \cdot \nabla (u^{n+1} - u^n) \cr
&\quad = - (u^{n+1} - u^n) - \nabla W \star (\rho^{n+1} - \rho^n) + \intr \psi(x-y) (u^n(y) - u^{n-1}(y)) \rho^{n+1}(y)\,dy\cr
&\qquad  - (u^n(x) - u^{n-1}(x)) \intr \psi(x-y) \rho^{n+1}(y)\,dy + \intr \psi(x-y)(u^{n-1}(y) - u^{n-1}(x)) (\rho^{n+1} - \rho^n)(y)\,dy,
\end{aligned}$$
respectively. Then multiplying \eqref{eqn_rho} by $\rho^{n+1} - \rho^n$ and integrating it over $\R^d$ gives
\bq\label{est_rhon}
\|(\rho^{n+1} - \rho^n)(\cdot,t)\|_{L^2}^2 \leq C\int_0^t \lt(\|(\rho^{n+1} - \rho^n)(\cdot,\tau)\|_{L^2}^2 + \|(u^n - u^{n-1})(\cdot,\tau)\|_{H^1}^2\rt)d\tau,
\eq
where $C > 0$ is independent of $n$. On the other hand, for $k=0,1$, we find
$$\begin{aligned}
&\frac12\frac{d}{dt}\intr |\nabla^k (u^{n+1} - u^n)|^2\,dx\cr
&\quad = -\intr \nabla^k (u^{n+1} - u^n) \nabla^k \lt((u^n - u^{n-1}) \cdot \nabla u^{n+1}\rt)dx\cr
&\qquad -\intr \nabla^k (u^{n+1} - u^n) \nabla^k \lt(u^{n-1} \cdot \nabla (u^{n+1} - u^n)\rt)dx\cr
&\qquad - \intr |\nabla^k (u^{n+1} - u^n)|^2\,dx -\intr \nabla^k (u^{n+1} - u^n)  \nabla^k (\nabla W \star (\rho^{n+1} - \rho^n)(x)) \,dx\cr
&\qquad +\intr \nabla^k (u^{n+1} - u^n)  \nabla_x^k \lt(\intr \psi(x-y) (u^n(y) - u^{n-1}(y)) \rho^{n+1}(y)\,dy\rt) \,dx\cr
&\qquad -\intr \nabla^k (u^{n+1} - u^n)  \nabla_x^k \lt( (u^n(x) - u^{n-1}(x)) \intr \psi(x-y) \rho^{n+1}(y)\,dy\rt) \,dx\cr
&\qquad +\intr \nabla^k (u^{n+1} - u^n)  \nabla_x^k \lt(\intr \psi(x-y)(u^{n-1}(y) - u^{n-1}(x)) (\rho^{n+1} - \rho^n)(y)\,dy\rt) \,dx  =: \sum_{i=1}^7 K_i,
\end{aligned}$$
where we easily estimate 
\[
\sum_{i=1}^3 K_i \leq C\| u^{n+1} - u^n\|_{H^1}^2+ C \|u^n - u^{n-1}\|_{H^1}^2.
\]
Here $C>0$ is independent of $n$. We next use the following estimates
$$\begin{aligned}
\lt|\intr  (u^{n+1} - u^n)(x) \cdot (\nabla W \star (\rho^{n+1} - \rho^n)(x)) \,dx\rt| 
&\leq C\|u^{n+1} - u^n\|_{L^2}\|\nabla W\|_{L^1}\|\rho^{n+1} - \rho^n\|_{L^2}\cr
&\leq C\|u^{n+1} - u^n\|_{L^2}\|\rho^{n+1} - \rho^n\|_{L^2}
\end{aligned}$$
and
\[
\lt|\intr \nabla (u^{n+1} - u^n)(x): (\nabla^2 W \star (\rho^{n+1} - \rho^n)(x)) \,dx\rt|\leq \|\nabla^2 W\|_{L^1}\|\nabla (u^{n+1} - u^n)\|_{L^2}\|\rho^{n+1} - \rho^n\|_{L^2}
\]
to have
$
K_4 \leq C\| u^{n+1} - u^n\|_{H^1}^2 + C\|\rho^{n+1} - \rho^n\|_{L^2}^2.
$
For the rest, if $k=0$, then
$$\begin{aligned}
K_5 &\leq  \|u^{n+1} - u^n\|_{L^2}\|\psi\|_{L^2}\|u^n - u^{n-1}\|_{L^2}\|\rho^{n+1}\|_{L^2}\leq C\|u^{n+1} - u^n\|_{L^2}^2 + C\|u^n - u^{n-1}\|_{L^2}^2,\cr
K_6 &\leq \|u^{n+1} - u^n\|_{L^2}\|u^n - u^{n-1}\|_{L^2}\|\psi\|_{L^2}\|\rho^{n+1}\|_{L^2} \leq C\|u^{n+1} - u^n\|_{L^2}^2 + C\|u^n - u^{n-1}\|_{L^2}^2,\cr
K_7 &\leq R\|\nabla u^{n-1}\|_{L^\infty} \|\psi\|_{L^1}\|u^{n+1} - u^n\|_{L^2}\|\rho^{n+1} - \rho^n\|_{L^2}\leq C\|u^{n+1} - u^n\|_{L^2}^2 + C\|\rho^{n+1} - \rho^n\|_{L^2}^2.
\end{aligned}$$
On the other hand, if $k=1$, we obtain
$$\begin{aligned}
K_5 &\leq  \|\nabla(u^{n+1} - u^n)\|_{L^2}\|\nabla \psi\|_{L^2}\|u^n - u^{n-1}\|_{L^2}\|\rho^{n+1}\|_{L^2}\leq C\|\nabla(u^{n+1} - u^n)\|_{L^2} + C\|u^n - u^{n-1}\|_{L^2}^2,\cr
K_6 &\leq \|\nabla(u^{n+1} - u^n)\|_{L^2}\lt(\|\nabla(u^n - u^{n-1})\|_{L^2}\|\psi\|_{L^2} + \|u^n - u^{n-1}\|_{L^2}\|\nabla \psi\|_{L^2}\rt) \|\rho^{n+1}\|_{L^2} \cr
&\leq C\|\nabla(u^{n+1} - u^n)\|_{L^2} + C\|u^n - u^{n-1}\|_{H^1}^2,\cr
K_7 &\leq \|\nabla(u^{n+1} - u^n)\|_{L^2} \lt(R\|\nabla u^{n-1}\|_{L^\infty}\|\nabla \psi\|_{L^1} + \|\psi\|_{L^1}\|\nabla u^{n-1}\|_{L^\infty} \rt)\|\rho^{n+1} - \rho^n\|_{L^2}\cr
&\leq C\|\nabla(u^{n+1} - u^n)\|_{L^2} + C\|\rho^{n+1} - \rho^n\|_{L^2}^2.
\end{aligned}$$
We now combine all of the above estimates to have
\[
\frac{d}{dt}\|u^{n+1} - u^n\|_{H^1}^2 \leq C\| u^{n+1} - u^n\|_{H^1}^2+ C \|u^n - u^{n-1}\|_{H^1}^2 + C\|\rho^{n+1} - \rho^n\|_{L^2}^2,
\]
and subsequently this yields
\[
\|(u^{n+1} - u^n)(\cdot,t)\|_{H^1}^2 \leq C\int_0^t \lt(\|(\rho^{n+1} - \rho^n)(\cdot,\tau)\|_{L^2}^2 + \|(u^n - u^{n-1})(\cdot,\tau)\|_{H^1}^2\rt)d\tau,
\]
where $C > 0$ is independent of $n$. This together with \eqref{est_rhon} asserts that $(\rho^n,u^n)$ is a Cauchy sequence in $\mc([0,T];L^2(\R^d)) \times \mc([0,T];H^1(\R^d))$. Interpolating this strong convergences with the above uniform-in-$n$ bound estimates gives
\[
\rho^n \to \rho \quad \mbox{in }\mc([0,T_*]; H^{s-1}(\R^d)), \quad u^n \to u \quad \mbox{in }\mc([0,T_*]; H^1(B(0,R))) \quad \mbox{as } n\to\infty,
\]
\[
\nabla u^n \to \nabla u \quad \mbox{in } \mc(\R^d \times [0,T_*]), \quad \mbox{and} \quad \nabla^2 u^n \to \nabla^2 u \quad \mbox{in } \mc([0,T_*];H^{s-2}(\R^d)) \quad \mbox{as } n\to\infty,
\]
due to $s > d/2+1$. We then use a standard functional analytic arguments, see for instances \cite[Section 2.1]{CK16}, to have that the limiting functions $\rho$ and $u$ satisfy the regularity in Theorem \ref{thm_local}. We easily show that the limiting functions $\rho$ and $u$ are solutions to \eqref{main_fluid2} in the sense of Definition \ref{def_strong2}.

We finally provide the uniqueness of strong solutions. Let $(\rho,u)$ and $(\tilde\rho, \tilde u)$ be the strong solutions obtained above with the same initial data $(\rho_0, u_0)$. Set $\Delta(t)$ a difference between two strong solutions:
\[
\Delta(t) := \|\rho(\cdot,t) - \tilde\rho(\cdot,t)\|_{L^2} + \|u(\cdot,t) - \tilde u(\cdot,t)\|_{H^1}.
\]
Then by using almost the same argument as above, we have
\[
\Delta(t) \leq C\int_0^t \Delta(s)\,ds \quad \mbox{with} \quad \Delta(0) = 0.
\]
This concludes that $\Delta(t) \equiv 0$ on $[0,T_*]$ and completes the proof.

%
%
%
%

\section*{Acknowledgments}
JAC was partially supported by EPSRC grant number EP/P031587/1 and the Advanced Grant Nonlocal-CPD (Nonlocal PDEs for Complex Particle Dynamics: Phase Transitions, Patterns and Synchronization) of the European Research Council Executive Agency (ERC) under the European Union's Horizon 2020 research and innovation programme (grant agreement No. 883363). YPC was supported by NRF grant (No. 2017R1C1B2012918), POSCO Science Fellowship of POSCO TJ Park Foundation, and Yonsei University Research Fund of 2019-22-021. 


\appendix

\section{Well-posedness of the particle system}\label{app_par}

In this appendix, we study the global existence and uniqueness of classical solutions to the particle system \eqref{main_par}--\eqref{ini_main_par}. 

Let us first consider the case with singular interaction potentials with $d\geq 2$. In this case, we can use the repulsive effect from the interaction forces, and this also enables us to have the uniqueness of solutions. 
\begin{theorem}\label{thm_par_ext} Let $d \geq 2$. Suppose that $\wt W$ is of the form \eqref{w_a} or \eqref{w_a2} and the confinement potential $V$ satisfies either $V \to +\infty$ as $|x| \to \infty$ or $\nabla_x V$ has linear growth as $|x| \to \infty$. If the initial data $x_0$ satisfy
\[
\min_{1 \leq i\neq j \leq N}|x_{i0} - x_{j0}| > 0.
\]
Then there exists a unique global smooth solution to the system \eqref{main_par}--\eqref{ini_main_par} with $\wt W$ instead of $W$ satisfying
\[
C \geq \max_{1 \leq i\neq j \leq N}|x_i(t) - x_j(t)| \geq \min_{1 \leq i\neq j \leq N}|x_i(t) - x_j(t)| > 0
\]
for $t \geq 0$, where $C>0$ is independent of $t$.
\end{theorem}
\begin{proof} For the proof, we first introduce the maximal life-span $T_0 = T(x_0)$ of the initial data data $x_0$ as
\[
T_0 := \sup\lt\{ s > 0 : \mbox{solution $(x(t), v(t))$ for the system \eqref{main_par} exists up to the time $s$} \rt\}.
\]
Then by the assumption and continuity of solutions, we get $T_0>0$. We now claim that $T_0 = \infty$ and for this it suffices to show that there is no collision between particles for all $t \geq 0$ and that particles cannot escape to infinity in finite time.

A straightforward computation yields
$$\begin{aligned}
\frac12\frac{d}{dt}\sum_{i=1}^N |v_i|^2& = -\gamma \sum_{i=1}^N |v_i|^2 - \sum_{i=1}^N v_i \cdot \nabla_x V(x_i)\cr
&\quad  - \frac1N\sum_{i\neq j}^N v_i \cdot \nabla_x \wt W(x_i - x_j) + \frac1N\sum_{i,j=1}^N \psi(x_i - x_j)(v_j - v_i) \cdot v_i
\end{aligned}$$
for $t \in [0,T_0)$. Note that
\[
\frac{d}{dt} \sum_{i=1}^N V(x_i) = \sum_{i=1}^N v_i \cdot \nabla_x V(x_i)
\]
and
\[
\frac1{2N}\frac{d}{dt}\sum_{i\neq j}^N \wt W(x_i - x_j) = \frac1{2N}\sum_{i\neq j}^N \nabla_x \wt W(x_i - x_j) \cdot (v_i - v_j) = \frac1N\sum_{i\neq j}^N \nabla_x \wt W(x_i - x_j) \cdot v_i,
\]
where we used $\nabla \wt W(-x) = -\nabla \wt W(x)$. Similarly, we also find
\[
\frac1N\sum_{i,j=1}^N \psi(x_i - x_j)(v_j - v_i) \cdot v_i = -\frac1{2N}\sum_{i,j=1}^N \psi(x_i - x_j)|v_j - v_i|^2.
\]
Combining all of the above estimates, we obtain
\[
\frac{d}{dt}\mathcal{F}^N(x,v) + \gamma \sum_{i=1}^N |v_i|^2 +\frac1{2N}\sum_{i,j=1}^N \psi(x_i - x_j)|v_j - v_i|^2 =0
\]
for $t \in [0,T_0)$, where $\mathcal{F}^N(x,v)$ denotes the discrete free energy given by
\[
\mathcal{F}^N(x,v) := \frac12\sum_{i=1}^N |v_i|^2 +  \sum_{i=1}^N V(x_i) + \frac1{2N}\sum_{i\neq j}^N \wt W(x_i - x_j).
\]
If $d=2$, then we have either 
\[
\frac1{2N}\sum_{i\neq j}^N \frac{1}{|x_i - x_j|^\alpha} \leq \mathcal{F}^N(x_0,v_0) \quad \mbox{or} \quad -\frac{1}{2N} \sum_{i \neq j} \log|x_i(t) - x_j(t)| \leq \mathcal{F}^N(x_0,v_0),
\]
where $\alpha \in (0,2)$. On the other hand, if $d \geq 3$, we obtain 
\[
\frac1{2N} \sum_{i\neq j} \frac{1}{|x_i(t) - x_j(t)|^\alpha} \leq \mathcal{F}^N(x_0,v_0)
\]
for all $t \in [0,T_0)$, where $\alpha \in (d-2,d)$. Since the right hand side of the above inequality is uniformly bounded in $t$, we conclude $T_0 = \infty$ for the case $d \geq 2$. An upper bound estimate of the distance between particles is a simple consequence of the uniform-in-time bound estimate of the free energy $\mathcal{F}^N$ due to the confinement potential whenever is present. If $V=0$, one can obtain that particles cannot escape to infinity in finite time as soon as $\nabla_x V$ has linear growth as $|x|\to\infty$. 
\end{proof}

\begin{remark} If the interaction and confinement potentials $W$ and $V$ are regular enough, i.e., $\nabla_x W \in \W^{1,\infty}(\R^d)$ and $\nabla_x V \in \W^{1,\infty}(\R^d)$, we have global-in-time existence and uniqueness of solutions by the standard Cauchy-Lipschitz theory. Moreover, an uniform-in-time bound of the distance between particles can also obtained due to the confinement potential if $V \to +\infty$ as $|x| \to \infty$.
\end{remark}

Let us finally comment on the one dimensional case. If $d=1$ and the interaction potential $\wt W$ is given by \eqref{w_a2}, then we apply Theorem \ref{thm_par_ext} to get the global unique classical solution and uniform-in-time bound estimate. If $W$ is given by the Coulomb potential, i.e., 
\bq\label{sgn}
W'(x) = \frac12 sgn(x), \quad \mbox{where} \quad sgn(x) := \left\{ \begin{array}{ll}
\displaystyle \frac{x}{|x|} & \textrm{if $x \neq 0$}\\[4mm]
0 & \textrm{if $x=0$}
  \end{array} \right..
\eq
Thus the interaction force $- W'$ is discontinuous, but bounded. In this sense, it is not so singular compared to the other cases. Since the velocity alignment force is regular, we can use a similar argument as in \cite[Proposition 1.2]{H14}, see also \cite{CCHS19,F88}, to have the following proposition.
\begin{proposition}\label{prop_1d} Let $d=1$. For any initial configuration $\mz^N(0)$, there exists at least one global-in-time solution to the system of \eqref{main_par} with \eqref{sgn} in the sense that $(x_i(t),v_i(t))$ satisfies the integral system:
$$\begin{aligned}
x_i(t) &= x_i(0) + \int_0^t v_i(s)\,ds, \quad i=1,\dots, N, \quad t >0,\cr
v_i(t) &= v_i(0) - \gamma \int_0^t v_i(s)\,ds - \int_0^t V'(x_i(s))\,ds - \frac1N \sum_{j\neq i}\int_0^t W'(x_i(s) - x_j(s))\,ds\cr
&\quad + \frac1N \sum_{j=1}^N \int_0^t \psi(x_i(s) - x_j(s))(v_j(s) - v_i(s))\,ds.
\end{aligned}$$
\end{proposition}
Even though Proposition \ref{prop_1d} does not provide the uniqueness of solutions, it is not necessary for the analysis of mean-field limit or mean-field/small inertia limit from the particle system \eqref{main_par} to the pressureless Euler system \eqref{main_fluid} or the aggregation equation \eqref{eq_agg}.

%
%
%
%

\end{document}